\pdfoutput=1
\documentclass[a4paper,10pt]{article}
\usepackage[T1]{fontenc}
\usepackage[utf8]{inputenc}
\usepackage[english,greek,frenchb,french]{babel}
\usepackage{url,csquotes}
\usepackage[hidelinks,hyperfootnotes=false]{hyperref}
\usepackage{float}
\usepackage{amsfonts,amssymb,enumerate}
\usepackage{amsthm}
\usepackage{amsmath}
\usepackage{mathtools,thmtools}
\allowdisplaybreaks[1]
\usepackage{graphicx}
\usepackage{bbm}
\usepackage{listings}
\usepackage{titling}
\usepackage{dsfont}
\usepackage{color}
\usepackage{enumitem}
\usepackage{footnote}

\newcommand{\E}{\mathbb{E}}
\newcommand{\cE}{\mathcal{E}}
    \newcommand{\Prb}{\mathbb{P}}
	
		\newcommand{\cG}{\mathcal{G}}

\newcommand{\cI}{\mathcal{I}}

\newcommand{\pp}{\mathfrak{p}}

		\newcommand{\fC}{\mathfrak{C}}
	
	\newcommand{\sR}{\mathbb{R}}
	\newcommand{\sN}{\mathbb{N}}
		\newcommand{\sS}{\mathbb{S}}
	\DeclareMathOperator{\Diam}{diam}
		
				\DeclareMathOperator{\Var}{Var}
			
			\DeclareMathOperator{\dis}{d}
				
				\DeclareMathOperator{\inte}{int}

				\DeclareMathOperator{\animals}{Animals}
					\DeclareMathOperator{\shell}{shell}

    \newcommand{\sZ}{\mathbb{Z}}
    \newcommand{\sC}{\mathcal{C}}
    \newcommand{\ep}{\varepsilon}    
    \newcommand{\ind}{\mathds{1}}
\declaretheorem[name=Theorem,within=section]{thm}
\declaretheorem[name=Proposition,numberlike=thm]{prop}
\declaretheorem[name=Definition,numberlike=thm]{defn}
\declaretheorem[name=Lemma,numberlike=thm]{lem}
\declaretheorem[name=Corollary,numberlike=thm]{cor}
\declaretheorem[name=Remark,numberlike=thm]{rk}

\newlength{\separationtitre}
\author{ Rapha\"el Cerf \thanks{
Ecole Normale Supérieure, PSL University, CNRS, DMA, 75005, Paris, France.}
\thanks{ \noindent Université Paris-Saclay, CNRS,  Laboratoire de mathématiques d'Orsay, 91405, Orsay, France.}\quad Barbara Dembin\thanks{\noindent ETH Zürich}
}
\date{}
\begin{document}
\newpage

 \selectlanguage{english}

\title{The time constant for Bernoulli percolation is Lipschitz continuous strictly above $p_c$ 
}

\maketitle

\begin{center}
\textit{Dedicated to the memory of Vladas Sidoravicius
\footnote{A few years ago, Vladas told R.C.: "Sometimes, you have to do a full
multiscale analysis only to get rid of a logarithm." This work is just a proof
that he was absolutely right.}}
\end{center}

\textbf{Abstract}: We consider the standard model of i.i.d. first passage percolation on $\sZ^d$ given a distribution $G$ on $[0,+\infty]$ ($+\infty$ is allowed). When \mbox{$G([0,+\infty])<p_c(d)$,} it is known that the time constant $\mu_G$ exists. We are interested in the regularity properties of the map $G\mapsto\mu_G$. We first study the specific case of distributions of the form $G_p=p\delta_1+(1-p)\delta_\infty$ for $p>p_c(d)$. In this case, the travel time between two points is equal to the length of the shortest path between the two points in a bond percolation of parameter $p$. We show that the function $p\mapsto \mu_{G_p}$ is Lipschitz continuous on every interval $[p_0,1]$, where $p_0>p_c(d)$.



\section{Introduction}
The model of first passage percolation was first introduced by Hammersley and Welsh \cite{HammersleyWelsh} as a model for the spread of a fluid in a porous medium. Let $d\geq 2$. We consider the graph $(\sZ^d,\E^d)$ having for vertices $\sZ^d$ and for edges $\E^d$ the set of the pairs of nearest neighbours in $\sZ^d$ for the Euclidean norm. To each edge $e\in\E^d$, we assign a random variable $t(e)$ with values in $\sR^+$ such that the family $(t(e),\,e\in\E^d)$ is independent and identically distributed with distribution $G$. The random variable $t(e)$ may be interpreted as the time needed for the fluid to cross the edge $e$.  We define a random pseudo-metric $T$ on this graph: for any pair of vertices $x$, $y\in\sZ^d$, the random variable $T(x,y)$ is the shortest time to go from $x$ to $y$. We are interested in the asymptotic behavior of the quantity $T(0,x)$ when $\|x\|$ goes to infinity. Under some assumptions on the distribution $G$, one can prove that $$\lim_{n\rightarrow\infty}\frac{1}{n}T(0,nx)=\mu_G(x)\,,$$
 where $\mu_G(x)$ is a deterministic constant depending only on the distribution $G$ and the point $x$. This result was proved by Cox and Durrett in \cite{CoxDurrett} in dimension $2$ under some integrability conditions on $G$, they also proved that $\mu_G$ is a semi-norm. Kesten extended this result to dimensions $d\geq 2$ in \cite{Kesten:StFlour}, and he proved that $\mu_G$ is a norm if and only if $G(\{0\})<p_c(d)$. The constant $\mu_G(x)$ may be seen as the inverse of the speed of spread of the fluid in the direction of $x$. It is usually called the time constant.

It is possible to extend this model by doing first passage percolation in a random environment.  We consider an i.i.d. supercritical bond percolation on the graph $(\sZ^d,\E^d )$.  Every edge $e\in\E^d$ is open with probability $p>p_c(d)$, where $p_c(d)$ denotes the critical parameter for this percolation. We know that there exists almost surely a unique infinite open cluster $\sC_p$ \cite{Grimmett99}. We can define the model of first passage percolation on the infinite cluster $\sC_p$. To do so, we consider a probability measure $G$ on $[0,+\infty]$ such that $G([0,\infty[)=p$. In this setting, the $p$-closed edges correspond to the edges with an infinite value while the infinite cluster made of the edges with finite passage times corresponds to the infinite cluster $\sC_p$ of a supercritical Bernoulli percolation of parameter $p$. The existence of a time constant for such distributions was first obtained in the context of a stationary integrable ergodic field by Garet and Marchand in \cite{GaretMarchand04} and was later shown for an independent field without any integrability condition by Cerf and Théret in \cite{cerf2016}. 

The question of the continuity of the map $G\mapsto \mu_G$ was first addressed in dimension $2$ with the article of Cox \cite{Cox}. He showed the continuity of this map under the following hypothesis of uniform integrability: if $G_n$ converges weakly towards $G$ and if there exists an integrable law $F$ such that, for all $n\in\sN$, $F$ stochastically dominates $G_n$, then $\mu_{G_n}$ converges towards $\mu_{G}$. In \cite{CoxKesten}, Cox and Kesten proved the continuity of this map in dimension $2$ without any integrability condition. Their idea was to consider a geodesic for the truncated passage times $\min(t(e),M)$ for some $M>0$, and then to avoid the clusters of  $G([0,M[)$-closed edges crossed by the geodesic, that is the clusters of edges with a passage time larger than some $M>0$. The clusters of $G([0,M[)$-closed edges intersecting the geodesic are then bypassed with a bypass included in their boundaries. Note that, by construction, the edges in the boundaries of $G([0,M[)$-closed clusters have a passage time smaller than $M$. Thanks to combinatorial considerations, Cox and Kesten were able to obtain a precise control on the length of these bypasses.  This idea was later extended to all the dimensions $d\geq 2$ by Kesten in \cite{Kesten:StFlour}, by taking a $M$ large enough such that the percolation of the edges with a passage time larger than $M$ is highly subcritical: for such a $M$, the size of the clusters of $p$-closed edges can be controlled. However, this idea does not work any more when we allow passage times to take infinite values.
In \cite{GaretMarchandProcacciaTheret}, Garet, Marchand, Procaccia and Théret proved the continuity of the map $G\mapsto\mu_G$ for general laws on $[0,+\infty]$ without any moment condition. More precisely, they proved the following. Let $(G_n)_{n\in\sN}$ and $G$ be probability measures on $[0,+\infty]$ such that $G_n$ converges weakly towards $G$, that is, for all continuous bounded functions $f:[0,+\infty]\rightarrow [0,+\infty[$, we have 
$$\lim_{n\rightarrow +\infty} \int _{[0,+\infty]} fdG_n= \int _{[0,+\infty]} fdG\, .$$
This convergence will be simply denoted by $G_n\overset{d}{ \rightarrow} G$.
Equivalently, we have that \smash{$G_n\overset{d}{ \rightarrow} G$} if and only if for any $t\in [0,+\infty]$ such that $x\mapsto G([x,+\infty])$ is continuous at $t$, we have $$\lim_{n\rightarrow +\infty}G_n([t,+\infty])=G([t,+\infty])\,.$$
If moreover $G_n([0,+\infty[)>p_c(d)$ for all $n\in\sN$, and $G([0,+\infty[)>p_c(d)$, then
\begin{align*}
\lim_{n\rightarrow \infty} \sup_{x\in\mathbb{S}^{d-1}} |\mu_{G_n}(x)-\mu_G(x)|=0\, ,
\end{align*}
where $\mathbb{S}^{d-1}$ is the unit sphere of $\sR^d$ for the Euclidean norm.

The regularity result on $\mu_G$ yields also some information on the limit shape, namely, on the way the limit shape changes under small perturbations. As mentioned in \cite{MR3729447} before Theorem 2.7:
\textit{ " If one could derive strong results in this direction, perhaps the establishment of various conjectures about the limit shape (e.g., curvature) could be made easier, or reduced to finding some special class of distributions for which the properties are explicitly derivable." } To be able to deduce a result on the stability of the curvature under small perturbations, we would need to obtain a regularity result on the second derivative of $G\mapsto\mu_{G}$. Therefore, there is still a lot of work to do in that direction. 
Our goal here is to improve the existing regularity result.  We wish to go beyond the mere continuity and to obtain a Lipschitz property. We consider first the specific case of distributions of the form $$G_p=p\delta_1+(1-p)\delta_\infty,\quad p>p_c(d)\,.$$  Let $\cG_p$ be the subgraph of $\sZ^d$  whose edges are open for the Bernoulli percolation of parameter $p$. The travel time for the law $G_p$ between two points $x$ and $y$ in $\sZ^d$ coincides with the chemical distance between $x$ and $y$, that is the graph distance between $x$ and $y$ in $\cG_p$. Namely, we define the chemical distance \smash{$D^{\cG_p}(x,y)$ }as the length of the shortest $p$-open path joining $x$ and $y$. As a corollary of the work of  Garet, Marchand, Procaccia and Théret in \cite{GaretMarchandProcacciaTheret}, we see that the map $p\mapsto\mu_{G_p}$ is continuous over $]p_c(d),1]$. In \cite{DembinRegularity}, Dembin obtained a better regularity property.
\begin{thm}[Theorem 1 in \cite{DembinRegularity}]
Let $p_0>p_c(d)$. There exists a constant $\kappa_0$ depending only on $d$ and $p_0$, such that
$$\forall p,q\in[p_0,1]\qquad\sup_{x\in\mathbb{S}^{d-1}}|\mu_{G_p}(x)- \mu_{G_q}(x)|\leq \kappa_0 |q-p||\log|q-p|| \,.$$
\end{thm}
\noindent To prove this theorem, Dembin used a renormalization process in which she controlled the scale of the renormalization. The renormalization was responsible for the presence of a logarithmic term. In this paper, we improve this result by proving that the function $p\mapsto \mu_{G_p}$ is in fact Lipschitz continuous.
\begin{thm}\label{thm1.2}
Let $p_0>p_c(d)$. There exists a constant $\kappa_0$ depending only on $d$ and $p_0$, such that 
$$\forall p,q\in[p_0,1]\qquad\sup_{x\in\mathbb{S}^{d-1}}|\mu_{G_p}(x)- \mu_{G_q}(x)|\leq \kappa_0 |q-p|\,.$$
\end{thm}

\noindent To fix the issues that were encountered in \cite{DembinRegularity}, we use a new approach. Our aim is to understand how the chemical distance in Bernoulli percolation depends upon the percolation parameter $p$. The key part of the proof lies in a multiscale modification of an arbitrary path. Let us fix two parameters $p,q$ such that $q>p>p_c(d)$. We couple two percolation configurations at level $p$ and $q$ in such way that a $p$-open edge is also $q$-open. We consider the geodesic $\gamma$ joining $0$ and $x\in\sZ^d$ for the bond percolation of parameter $q$. Some of the edges in $\gamma$ are $p$-closed, we want to build upon this path a $p$-open path. To do so, we need to bypass the $p$-closed edges in $\gamma$. Roughly speaking, the idea is to prove that, for $\|x\|$ large enough, with high probability, the average size of a bypass is smaller than a constant $C$ and that the number of edges to bypass in $\gamma$ is at most $(q-p)\|x\|$. Therefore, with high probability the total length of the bypasses is less than $C(q-p)\|x\|$. Whereas in \cite{DembinRegularity}, all the edges were bypassed at the same scale, here we use a multiscale renormalization and each edge is bypassed at the appropriate scale. The crucial point is to perform each bypass at an adequate scale and to pay the right price for it. By properly choosing the different scales of the renormalization process, we can build a family of shells $(\shell(e))_{e\in\gamma}$ made of good boxes at scale $1$ such that the total cardinality of the shells $\sum_{e\in\gamma} |\shell(e)|$ is at most $C \|x\|$ with high probability. These shells of good boxes will possess all the desired properties to build $p$-open bypasses of edges in $\gamma$. The shells are built without revealing the $p$-states of the edges in $\gamma$ so that they are independent of the $p$-states of the edges in $\gamma$. In the end, we will not use all the shells but only the shells associated to $p$-closed edges in $\gamma$. In the coupling, the probability that a $q$-open edge is $p$-closed is $q-p$.  Therefore, we expect that the total length of the bypasses  $$\sum_{e\in\gamma} |\shell(e)|\ind_{e\text { is $p$-closed}}$$ is at most $C(q-p)\|x\|$. 

We did not manage to prove this result by using a simple renormalization process. In remark \ref{rkmultiscale}, we explain the technical difficulties that we encountered and why we think that a multiscale renormalization is necessary.

We tried to extend theorem \ref{thm1.2} to general distributions, but we did not obtain a satisfactory result. We only managed to prove the result stated next. Its formulation is quite heavy, and a main problem is that this result does not yield the mere continuity of the map $G\mapsto\mu_G$,
so we will only give a sketch of its proof in section \ref{sec:gendib}.

Let $p_1<p_c(d)$, $p_0>p_c(d)$, $M>0$, $\ep_0>0$ and $\ep\mapsto\delta(\ep)$ be a non-decreasing function. We define $\fC_{p_0,p_1,M,\ep_0,\delta}$ as 
$$\fC_{p_0,p_1,M,\ep_0,\delta}=\left\{\,\begin{array}{c} G \text{ distribution on $[0,+\infty]$} \,: \, G(\{0\})\leq p_1,\\G([0,+\infty[)>p_0, \,\forall \ep<\ep_0\quad G(]0,\ep])\leq \delta(\ep), \\G([0,M])\geq (1-\frac{\delta_0(p_0)}{2})G([0,+\infty[)\end{array}\,\right\}\,, $$
where $\delta_0=\delta_0(p_0)$ is a positive constant depending on $p_0$ and $d$.
For $F$ a distribution on $[0,+\infty]$, we denote by $\overline{F}$ the distribution $F$ conditioned to $[0,+\infty[$, defined by
$$\forall x\in\sR^+ \qquad\overline{F}([0,x])=\frac{F([0,x])}{F([0,+\infty[)}\,$$
and by $\overline{F}^{-1}$ the pseudo inverse of $\overline{F}$, defined by
 $$\forall t\in[0,1]\qquad \overline{F}^{-1}(t)=\inf\big\{\,x\in\sR:\, \overline{F}(x)\geq t\,\big\}\,.$$
The map $G\mapsto \mu_G$ is Lipschitz continuous on $\fC_{p_0,p_1,M,\ep_0,\delta}$ in the following sense. There exists a constant $\kappa$ depending on the parameters of the class $\fC_{p_0,p_1,M,\ep_0,\delta}$ such that
\begin{align*}\forall F,G\in \fC_{p_0,p_1,M,\ep_0,\delta}\qquad\sup_{x\in\sS^{d-1}}|\mu_{G}(x)-\mu_{F}(x)|\leq \kappa\big(\big|F(\{+\infty\})-G(\{+\infty\})\big|\\\hfill+\sup_{t\in[0,1]}\big|\overline{F}^{-1}(t)-\overline{G}^{-1}(t)\big|\big)\,.
\end{align*}

In order to understand better where this class of distributions comes from, let us consider two distributions $G$ and $F$ on $[0,+\infty[$ and the standard coupling of these two distributions, in which a uniform random variable on $[0,1]$ is associated to each edge. Let us consider the geodesic $\gamma_G$ for the distribution $G$ between $0$ and $x\in\sZ^d$. The time $T_F(0,x)$ to go from $0$ to $x$ for the distribution $F$ is bounded from above by $$T_G(0,x)+|\gamma_G|\sup_{t\in[0,1]}|F^{-1}(t)-G^{-1}(t)|\,.$$ Conversely, the same inequality holds for $T_G(0,x)$. 
We seek a class $\fC$ of distributions on which the size of $\gamma_G$ is uniformly bounded from above, \textit{i.e.}, for which there exists a constant $C$ such that, for all distributions $G$ in $\fC$, we have $|\gamma_G|\leq C\|x\|$ with high probability when $\|x\|$ goes to infinity. The inequality $|\gamma_G|\leq C\|x\|$ ensures that
$$\frac{|T_G(0,x)-T_F(0,x)|}{\|x\|}\leq C\sup_{t\in[0,1]}|F^{-1}(t)-G^{-1}(t)|\,.$$
 When we consider distributions that may take infinite values, we add another difficulty. For some edges in $\gamma_G$ the passage time for the law $F$ may be infinite. To overcome this issue, we apply the same strategy as in the proof of theorem \ref{thm1.2}: we bypass these edges with edges of passage time smaller than some constant $M$ for the law $F$. The number of edges that we need to bypass is of order at most $|F(\{+\infty\})-G(\{+\infty\})||\gamma_G|$ and the average size of a bypass is constant. This accounts for the term $|F(\{+\infty\})-G(\{+\infty\})|$ in the statement of the result.

We do not claim that the constraints on the distributions given by the class $\fC$ are optimal. However, for each condition, we can exhibit a family of distributions for which it is unclear whether we can obtain a uniform control on the size of the geodesic:
\begin{itemize}[leftmargin=*]
\item The family $\big((p_c-1/n)\delta_0+(1-p_c+1/n)\delta_M\big)_{n\geq 1}$ for some large $M>0$ accounts for the condition $G(\{0\})<p_1$.
\item The family $\big(p_c\delta_{1/n}+(1-p_c)\delta_M\big)_{n\geq 1}$ for some large $M>0$ accounts for the condition $\forall \ep<\ep_0,\,G(]0,\ep])\leq \delta(\ep)$.
\item The family $\big((p_c+1/n)\delta_{1}+(1-p_c-1/n)\delta_\infty\big)_{n\geq 1}$ accounts for the condition $G([0,+\infty[)>p_0$.
\item The family $\big(p_c\delta_{1}+(1-p_c)\delta_n\big)_{n\geq 1}$ accounts for the condition $$G([0,M])\geq (1-\delta_0/2)G([0,+\infty[)\,.$$
\end{itemize}
At $p_c$, we know that the size of a geodesic is super linear in dimension $2$ \cite{2016DamronTang} and the nature of the problem is very different.

Here is the structure of the paper. In section \ref{defs}, we introduce some definitions and preliminary results. The section \ref{renormalization} presents the multiscale renormalization process and the construction of the shells. In this section, we explain how we modify a $q$-open path to turn it into a $p$-open path with a control on the length of the bypasses. In section \ref{goodbox}, we derive probabilistic estimates on the total size of the shells. Finally, in section~\ref{estimate}, we prove theorem \ref{thm1.2} and sketch the proof for the result for general distributions.

\section{Definitions and preliminary results}\label{defs}
Let $d\geq 2$. Let $x=(x_1,\dots,x_d)\in\mathbb{R}^d$, we define $$\|x\|_1=\sum_{i=1}^d |x_i|,\quad\|x\|_2=\sqrt{\sum_{i=1}^d x_i^2}\quad\text{and}\quad\|x\|_\infty=\max_{1\leq i\leq d} |x_i|\,.$$ 
Let $A$ and $B$ be two finite subsets of $\sZ^d$, we define the distance $\dis(A,B)$ between the sets $A$ and $B$ as
$$\dis(A,B)=\inf\Big\{\,\|a-b\|_1,\, a\in A,\, b\in B\,\Big\}\,.$$
We extend this definition to sets of edges $A, B\subset\E^d$. Let $\widetilde{A}$ (respectively $\widetilde{B}$) be the set of the endpoints of the edges in $A$ (respectively $B$), we set \[\dis(A,B)=\dis(\widetilde{A},\widetilde{B})\,.\]

Let $G\subset \sZ^d$ and $x,y\in G$. We say that the sequence $\gamma=(v_0,\dots,v_n)$ is a $*$-path from $x$ to $y$ in $G$ if $v_0=x$, $v_n=y$ and for all $i\in\{1,\dots,n\}$, $v_i\in G$ and $\|v_i-v_{i-1}\|_\infty=1$. We say that $x$ and $y$ are $*$-connected in $G$ if such a path exists. 
Let $\cG$ be a subgraph of $(\sZ^d,\E^d)$ and let $x,y\in\cG$. A path $\gamma$ from $x$ to $y$ in $\cG$ is a sequence $\gamma=(v_0,e_1,\dots,e_n,v_n)$ such that $v_0=x$, $v_n=y$ and for all $i\in\{1,\dots,n\}$, the edge $e_i=\langle v_{i-1},v_i\rangle$ belongs to $\cG$. The length of such a path is $n$ and it is denoted by $|\gamma|$. We say that $x$ and $y$ are connected in $\cG$ if such a path exists. We define the chemical distance between $x$ and $y$ in $\cG$ by $$D^{\cG}(x,y)=\inf\Big\{\,|r|: \text{$r$ is a path from $x$ to $y$ in $\cG$}\,\Big\}\,.$$ If $x$ and $y$ are not connected in $\cG$, then we have $D^{\cG}(x,y)=\infty$. In what follows, the graph $\cG$ will be the subgraph $\cG_p$ of $\sZ^d$ whose edges are open for the Bernoulli percolation of parameter $p>p_c(d)$.
To get around the fact that the chemical distance can take infinite values, we introduce a regularized chemical distance. We denote by $\sC_p$ the unique infinite connected component of $\cG_p$. Let $\sC$ be a subset of $\sC_p$, we define $\widetilde{x}^\sC$  as the vertex of $\sC$ which minimizes $\|x-\widetilde{x}^\sC\|_1$, with a deterministic rule to break ties. Typically, we will take for $\sC$ the infinite cluster of a configuration of Bernoulli percolation with a parameter smaller than $p$ so that $\sC\subset\sC_p\subset \cG_p$ and therefore

$$D^{\cG_p}(\widetilde{x}^\sC,\widetilde{y}^\sC)\leq D^{\sC}(\widetilde{x}^\sC,\widetilde{y}^\sC)<\infty\,.$$ 
We can define the regularized time constant as in \cite{GaretMarchand10} or as a special case of \cite{cerf2016}.

\begin{prop}\label{convergence}Let $p>p_c(d)$. There exists a deterministic function $\mu_p:\sZ^d\rightarrow [0,+\infty[$ such that, for every $p_0\in(p_c(d),p]$,
\begin{align*}
\forall x\in\sZ^d\qquad\lim_{n\rightarrow \infty} \frac{D^{\sC_{p}}(\widetilde{0}^{\sC_{p_0}},\widetilde{nx}^{\sC_{p_0}})}{n}=\mu_p(x)\qquad\text{a.s. and in $L^1$.}
\end{align*}
\end{prop}
\noindent It is important to check that $\mu_p$ does not depend on $\sC_{p_0}$, the infinite cluster we use to regularize the chemical distance. This is done in lemma 2.11 in \cite{GaretMarchandProcacciaTheret}. As a corollary, we obtain the monotonicity of the map $p\mapsto \mu_p$, see lemma 2.12 in \cite{GaretMarchandProcacciaTheret}.
\begin{cor}\label{decmu}For all $p_c(d)<p\leq q$, we have
$$\forall x\in\sZ^d\qquad\mu_p(x)\geq\mu_q(x)\,.$$
\end{cor} 
\noindent The following lemma is an improvement of the result of Antal and Pisztora in \cite{AntalPisztora} controlling the probability that two connected points have a too large chemical distance. In the original result, the constants depend on $p$, the adaptation of the proof is written in lemma 4.2. in \cite{DembinRegularity}.
\begin{lem}\label{AP}
Let $p_0>p_c(d)$. There exist $\beta=\beta(p_0)\geq 1$, $\widehat{A}=\widehat{A}(p_0)$ and $\widehat{B}=\widehat{B}(p_0)>0$ such that, for all $p\geq p_0$,
\begin{align*}
\forall x \in\sZ^d\qquad\Prb(\beta\|x\|_1\leq D^{\sC_p}(0,x)<+\infty)\leq \widehat{A}\exp(-\widehat{B}\|x\|_1)\, .
\end{align*}
\end{lem}
\section{Renormalization}\label{renormalization}

In this section, we present the multiscale renormalization process, the construction of the shells and the method to build bypasses.  Let $q>p\geq p_0>p_c(d)$ be fixed. We consider two configurations of Bernoulli percolation on the edges of $\sZ^d$, one with parameter $p$ and one with parameter $q$.

\subsection {Definition of the renormalization process} For a positive integer $N$, we define the $N$-box $$B_N=\left[-\frac{N}{2},\frac{N}{2}\right[^{\,d}\cap \sZ^d$$ and the family of translated $N$-boxes:
$$\forall \textbf{i}\in\sZ^d\qquad B_N(\textbf{i})=\textbf{i}N+B_N\,.$$
The lattice $\sZ^d$ is the disjoint union of this family: $\sZ^d=\sqcup_{\textbf{i}\in\sZ^d}B_N(\textbf{i})$. We introduce also larger boxes that will help us to link $N$-boxes together. We define
$$\forall \textbf{i}\in \sZ^d \qquad B'_N(\textbf{i})=\textbf{i}N+B_{3N}.$$
Let $B$ be a box. A connected cluster $C$ of the configuration restricted to $B$ is said to be crossing, if for any pair of opposite faces of $B$, there is an open path in $C$ connecting them.
We define the diameter of a cluster $C$ as 
$$\Diam(C):=\max_{\substack{ x,y\in C}}\|x-y\|_\infty\, .$$
We construct the multiscale renormalization process by induction. Let $(l_k)_{k\in\sN^*}$ with $l_1\geq 3$ be an increasing sequence of integers which will be specified later.  This sequence will define the successive scales. For a positive integer $k$, we define $N_k$ as $$N_k=l_1\cdots l_k\,.$$
We also define the box $\underline{B}_{k+1}$ of side length $l_{k+1}$ at scale $k$ that is made of sites at scale $k$. More precisely, for $\textbf{i}\in\sZ^d$, we define
$$\underline{B}_{k+1}(\textbf{i})=B_{l_{k+1}}(\textbf{i})\quad\qquad\text{and}\qquad\quad\underline{B}'_{k+1}(\textbf{i})=B'_{l_{k+1}}(\textbf{i})\,.$$
Thus, we have $|\underline{B}_{k+1}(\textbf{i})|=l_{k+1}^d$ and $|B_{N_{k+1}}(\textbf{i})|=N_{k+1}^d$.
The definition of what a good box is will differ at scale $1$ and at larger scales. We first have to define what a good box is at scale $1$. To do so, we list the properties that a good box should have to ensure that we can build a suitable modification of a path. We have to keep in mind that all the properties must occur with probability converging to $1$ when $N_1=l_1$ goes to infinity. Let $\beta=\beta(p_0)$ be the constant defined in lemma \ref{AP}.
\begin{defn}\label{defgoodbox}Let $\beta>0$ be a constant that will be defined later. We say that the site $\textbf{i}$ is $(p,q)$-good at the scale $1$ if the following events occur:
\begin{enumerate}[label=(\roman*)]
\item There exists a unique $p$-cluster $\sC$ in $B'_{N_1}(\textbf{i})$ with diameter larger than $N_1$;
\item This $p$-cluster $\sC$ is crossing for each of the $3^d$ $N_1$-boxes included in $B'_{N_1}(\textbf{i})$;
\item For all $x,y\in B'_{N_1}(\textbf{i})$, if $x$ and $y$ belong to $\sC$, then $D^{\cG_p}(x,y)\leq 12\beta N_1$;
\item If $\gamma$ is a $q$-open path in $ B'_{N_1}(\textbf{i})$ such that $|\gamma|\geq N_1$, then $\gamma$ and the $p$-cluster $\sC$ in $ B'_{N_1}(\textbf{i})$ have a vertex in common.
\end{enumerate}
The cluster $\sC$ is called the crossing $p$-cluster of the $(p,q)$-good site $\textbf{i}$.
We say that the box $B_{N_1}(\textbf{i})$ is $(p,q)$-good at scale $1$ if the site $\textbf{i}$ is $(p,q)$-good at scale $1$.
\end{defn}

\noindent We introduce next the notion of a cluster of bad sites.
\begin{defn} Let $k\geq 1$. Let us assume that we have already defined what a $(p,q)$-good box is at scale $k$.
For a $(p,q)$-bad site $\textbf{j}\in\sZ^d$ at scale $k$, we denote by $C ^{(k)}(\textbf{j})$ the connected cluster of the $(p,q)$-bad sites at scale $k$ which contains $\textbf{j}$ (if $\textbf{j}$ is a $(p,q)$-good site, then $C ^{(k)}(\textbf{j})$  is empty). Equivalently, the cluster $C ^{(k)}(\textbf{j})$ is the set of the vertices of $\sZ^d$ which are connected to $\textbf{j}$ by a path visiting only sites which are bad at scale $k$.
\end{defn}
\noindent We define now by induction what is a good site at scale $k$.
\begin{defn}\label{defgoodbigbox} Let $k\geq 1$. Let us assume that we have already defined what a $(p,q)$-good site is for scales from $1$ to $k$. Let $\textbf{i}\in\sZ^d$. We say that the site $\textbf{i}$ is $(p,q)$-good at scale $k+1$ if  $$\forall\textbf{j}\in \underline{B}'_{k+1}(\textbf{i})\qquad|C ^{(k)}(\textbf{j})|\leq l_{k+1}\,.$$
We say that the box $\underline{B}_{k+1}(\textbf{i})$ is $(p,q)$-good at scale $k+1$ if the site $\textbf{i}$ is $(p,q)$-good at scale $k+1$.
\end{defn}
\noindent To abbreviate, we will often say good site instead of $(p,q)$-good site.
On the grid $\sZ^d$, we use the standard definition of closest neighbour, \textit{i.e.}, we say that $x$ and $y$ are neighbours if $\|x-y\|_1=1$. Let $C$ be a connected set of sites of $\sZ^d$, we define its exterior vertex boundary 
$$\partial_v C =\left\{\begin{array}{c} \textbf{i}\in\sZ^d\setminus C: \text {\textbf{i} has a neighbour in C and is connected } \\ \text{to infinity by a path in $\sZ^d\setminus C$} \end{array} \right\}.$$
The set $\partial_v C$ is not $\sZ^d$-connected in general, however it is $*$-connected (see for instance lemma 2 in \cite{Timar}). Besides, we have $$|\partial_v C|\leq 2d|C|\,.$$ 
We adopt the convention that $\partial_v C ^{(k)} (\textbf{i})=\{\textbf{i}\}$ when $\textbf{i}$ is a good site at scale $k$.
We shall define a multiscale site percolation process given by the states of the boxes at the different scales. Note that, on a given scale, the states of the boxes are not independent but there is a short range dependence. 

\subsection{Construction of the detours}
 
We will consider different couplings between the percolation processes with parameters $p$ and $q$. These couplings are variants of the usual coupling built with the help of i.i.d. random variables uniformly distributed on $[0,1]$. For the time being, we do not specify which coupling we use. The most important property of the coupling is that a $p$-open edge is always $q$-open. Let us consider a $q$-open path $\gamma$ between $0$ and a point $x\in\sZ^d$. Some edges in $\gamma$ might be $p$-closed and our goal is to bypass these $p$-closed edges. To each edge $e$ in $\gamma$ we will associate a shell made of good boxes at scale $1$ such that the edge $e$ lies in the interior of the shell. The properties of the good boxes will guarantee that the edge $e$ can be bypassed by a $p$-open path lying in the internal boundary of its associated shell. To control the lengths of the bypasses, we shall bound from above the total size of the required shells, depending on the bad sites that $\gamma$ crosses.
Let us first rigorously define what a shell is. Let $C$ be a $*$-connected set. We define the interior $\inte(C)$ of $C$ by
$$\inte( C)=\left\{ \textbf{i}\in\sZ^d\setminus C:\begin{array}{c} \text {\textbf{i} is not connected to infinity} \\ \text{ by a $\sZ^d$-path in $\sZ^d\setminus C$} \end{array} \right\}.$$
\begin{defn}
Let $e\in\E^d$. A set $C$ of $*$-connected good boxes at scale $1$ is a shell for $e$ if it satisfies $$\partial_v \inte(C)=C\quad \text{and}\quad e\in\bigcup_{\textbf{i}\in\inte(C)}(B_{N_1}(\textbf{i})\cap\E^d)\,.$$
\end{defn}
\noindent The condition $\partial_v\inte(C)=C$ says that $C$ is indeed a sort of shell. The second condition says that $e$ is in the interior of the shell.

 For $k\geq 1$ and a path $\gamma$, we denote by $n_k(\gamma)$ the number of bad boxes at scale $k$ that $\gamma$ crosses, \textit{i.e.},
$$n_k(\gamma)=\left|\left\{\textbf{i}\in \sZ^d,\, \textbf{i}\text{ is bad at scale $k$ and }\gamma\cap B_{N_k}(\textbf{i})\neq \emptyset\right\}\right|\,.$$ 
The following proposition builds shells for edges in a path $\gamma$ and bounds from above the sum of the squares of the sizes of these shells in function of the number of the bad sites that $\gamma$ visits at every scale. Note that we do not build shells for edges around the extremities of $\gamma$. Due to technical details, for the edges at the extremities of $\gamma$, we cannot guarantee that we can build shells such that $\gamma$ enters and exits each shell at least once. 
\begin{prop}[Construction of the shells]\label{propconstr} Let $p_c(d)<p<q$. Let $x\in\sZ^d$. Let us assume that $0,x\in\sC_q$. We consider a $q$-open path $\gamma$ between $0$ and $x$. Let $M=M(\gamma)$ be the smallest positive integer such that $n_M(\gamma)=0$. We set $$\overline{\gamma}=\gamma \setminus (B_{4N_M}\cup (x+B_{4N_M}))\,.$$ 
(If $M$ is too large or infinite, then $\overline{\gamma}$ is empty). To each edge $e\in\overline{\gamma}$, we can associate a shell $\shell(e)$ such that we have $$\dis\Bigg(\bigcup_{\textbf{i}\in\shell(e)}B_{N_1}(\textbf{i}),\{e\}\Bigg)\geq (14\beta+2d) N_1\,,$$$$\dis\Bigg(\bigcup_{\textbf{i}\in\inte(\shell(e))}B_{N_1}(\textbf{i}),\{0,x\}\Bigg)\geq N_1\,,$$
\begin{align*}
\sum_{e\in\overline{\gamma}}|\shell(e)|^2&\leq 3^{2d}4d^2\left((3d)^4|\gamma|N_{3}^2N_2^{2d}+ \sum_{k=3}^{M-1} n_k(\gamma)N_{k+1}^2N_k^{3d}(3d)^{2k}d\right)\,. 
\end{align*}
\end{prop}
\begin{proof}[Proof of proposition \ref{propconstr}] Let $\gamma$ be a $q$-open path joining $0$ and $x$ and let $M$ and $\overline{\gamma}$ be defined as in the statement of the proposition.  Let $e\in\overline{\gamma}$. For $k\geq 1$, we denote by $\textbf{e}_k$ the site of $\sZ^d$ such that the box $B_{N_k}(\textbf{e}_k)$ contains the smallest extremity of $e$ in the lexicographic order. Let us assume that there exists an integer $ k\in\{2,\dots, M-1\}$ such that $\textbf{e}_3,\dots,\textbf{e}_k$ are $(p,q)$-bad at their respective scales but $\textbf{e}_{k+1}$ is $(p,q)$-good at the scale $k+1$ (if $\textbf{e}_3$ is good then $k=2$).  We define $B'$, $\Lambda_k$, $\widetilde{\Lambda}_k$ and $\overline{\Lambda}_k$ by
$$B'=\big\{\textbf{i}\in\sZ^d:\,\|\textbf{i}-\textbf{e}_k\|_{\infty}\leq 1\,\big\},\quad\Lambda_k=\bigcup_{\textbf{i}\in\partial_v B'}C ^{(k)}(\textbf{i})\cup B',$$
$$\quad\widetilde{\Lambda}_k=\bigcup_{\textbf{i}\in \Lambda_k}\underline{B}_k(\textbf{i})\qquad\text{and}\qquad\overline{\Lambda}_k=\bigcup_{\textbf{i}\in\Lambda_k}B_{N_k}(\textbf{i})\,.$$
The sets $B'$ and $\Lambda_k$ are made of sites at scale $k$. The set $\widetilde{\Lambda}_k$ is made of sites at scale $k-1$. The set $\overline{\Lambda}_k$ is made of sites belonging to the initial lattice $\sZ^d$.
Since $\textbf{e}_{k+1}$ is a good site at scale $k+1$ and $\partial_v B' \subset \underline{B}'_{k+1}(\textbf{e}_{k+1})$, we have 
$$\forall \textbf{i}\in\partial_v B'\qquad|C ^{(k)}(\textbf{i})|\leq l_{k+1}\,,$$ and, using the fact that $|\partial _v B'|\leq 2d |B'|$, we obtain $$ |\Lambda_k|\leq (2dl_{k+1}+1)|B'|\leq 3^{d+1}dl_{k+1}\,.$$
 Moreover, we claim that the set $\partial_v\Lambda_k$ is made of good sites at scale $k$. Let $\textbf{i}\in\partial_v\Lambda_{k}$. Since $$\partial_v\Lambda_k\subset \bigcup_{\textbf{l}\in\partial_v B'}\partial_v C ^{(k)}(\textbf{l})\,,$$ there exists $\textbf{l}\in\partial_v B'$ such that $\textbf{i}\in\partial_vC ^{(k)}(\textbf{l})$ and so $\textbf{i}$ is indeed a good site at scale $k$.

Let us assume that $l_2\geq (14\beta+2d)$. By construction of $B'$ and since $e\in\overline{\Lambda}_k$, we have 
\begin{align}\label{eqnum:1}
\dis(\{e\},\sZ^d\setminus \overline{\Lambda}_k)\geq N_k\geq N_2\geq (14\beta+2d) N_1\,.
\end{align}
The set $\Lambda_k$ is included in the box $B_{2(l_{k+1}+2)}(\textbf{e}_k)$ of sites at scale $k$. Thus, we have
\begin{align*}
\dis(\{e\},\partial_v \overline{\Lambda}_k)&\leq \sup_{x\in B_{2(l_{k+1}+2)N_k}}\dis(x,\sZ^d\setminus  B_{2(l_{k+1}+2)N_k})\\&=\dis(0,\sZ^d\setminus B_{2(l_{k+1}+2)N_k}(\textbf{0})) \leq (l_{k+1}+2)N_k=N_{k+1}+2N_k\,.
\end{align*}
Moreover, as $e\in\overline{\gamma}$, we have $\dis(\{e\},\{0,x\})\geq 2 N_M$
and using the fact that $e\in\overline{\Lambda}_k$, we have
$$\dis(\{e\},\{0,x\})\leq \dis(\{e\},\partial_v \overline{\Lambda}_k)+\dis(\partial_v \overline{\Lambda}_k,\{0,x\})\leq \dis(\{e\},\partial_v \overline{\Lambda}_k)+\dis(\overline{\Lambda}_k,\{0,x\})\,,$$
thus $$\dis(\overline{\Lambda}_k,\{0,x\})\geq 2N_M-N_{k+1}-2N_k\,.$$
\begin{figure}[!ht]
\def\svgwidth{1\textwidth}
 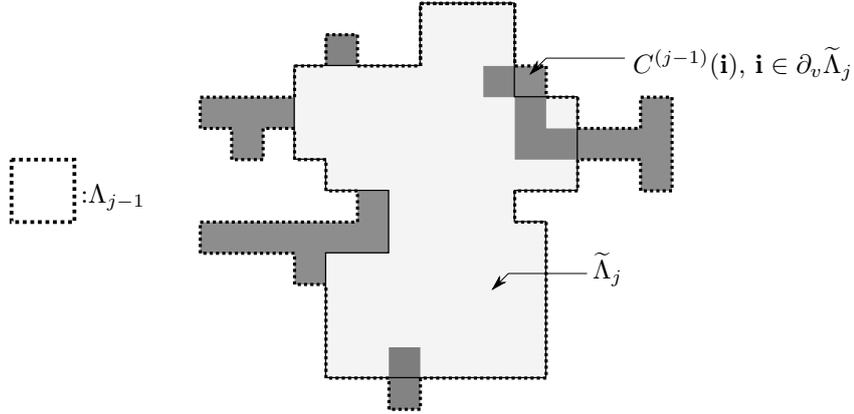
 \caption[fig1]{\label{fig1}Construction of $\Lambda_{j-1}$}
\end{figure}
We define next iteratively a sequence of sets $\Lambda_k,\widetilde{\Lambda}_k,\dots,\Lambda_1,\widetilde{\Lambda}_1$. Let $j\geq 2$. Let us assume that we have already defined $\Lambda_k,\widetilde{\Lambda}_k,\dots,\Lambda_j,\widetilde{\Lambda}_j$. We define $\Lambda_{j-1}$, $\widetilde{\Lambda}_{j-1}$ (see Figure \ref{fig1}) and $\overline{\Lambda}_{j-1}$ by
$$\Lambda_{j-1}=\bigcup_{\textbf{i}\in \partial_v\widetilde{\Lambda}_j}C ^{(j-1)}(\textbf{i}) \cup\widetilde{\Lambda}_j,\quad\widetilde{\Lambda}_{j-1}=\bigcup_{\textbf{i}\in\Lambda_{j-1}}\underline{B}_{j-1}(\textbf{i})$$ and$$\quad\overline{\Lambda}_{j-1}=\bigcup_{\textbf{i}\in\Lambda_{j-1}}B_{N_{j-1}}(\textbf{i})\,.$$
Let us prove by decreasing induction that, for $1\leq j\leq k$, the following holds:
\begin{enumerate}[label=(\roman*)]
\item We have $|\Lambda_j|\leq 3^d(3d)^{k+1-j}l_{k+1}\left(N_k/N_j\right) ^{d+1}$.
\item \label{cond:2} The set $\partial_v\Lambda_j$ is made of good sites at scale $j$.
\item \label{cond:3}We have $\dis(\overline{\Lambda}_j,\{0,x\})\geq 2N_M-2N_k -\sum_{l=j+1}^{k+1}N_l$.
\end{enumerate}
These properties are true for $k$. Let us now assume that these properties hold for some integer $2\leq j\leq k$. Let $\textbf{i}\in \partial_v\widetilde{\Lambda}_j$ be a site such that $C ^{(j-1)}(\textbf{i})\neq \emptyset$. There exists $\textbf{l}\in\widetilde{\Lambda}_j$ such that $\textbf{i}$ is a neighbour of $\textbf{l}$. Let $\underline{\textbf{l}}$ (respectively $\underline{\textbf{i}}$) be such that $\textbf{l}\in \underline{B}_j(\underline{\textbf{l}})$ (respectively $\textbf{i}\in \underline{B}_j(\underline{\textbf{i}})$). Since $\textbf{i}\notin \widetilde{\Lambda}_j$ and $\textbf{l}\in \widetilde{\Lambda}_j$, then we have $\underline{\textbf{i}}\notin \Lambda_j$ and $\underline{\textbf{l}}\in \Lambda_j$. Since the sites $\underline{\textbf{i}}$ and $\underline{\textbf{l}}$ are neighbours, it follows that $\underline{\textbf{i}}\in \partial_v \Lambda_j$. Thanks to $(ii)$, the site $\underline{\textbf{i}}$ is a good site at scale $j$ and so we have $|C ^{(j-1)}(\textbf{i})|\leq l_j$ and
$$|\Lambda_{j-1}|\leq (2dl_j+1)|\widetilde{\Lambda}_j|\leq 3dl_j^{d+1}|\Lambda_j|\,.$$
Iterating this inequality, we obtain property $(i)$:
\begin{align*}
|\Lambda_{j-1}|&\leq 3^d(3d)^{k+2-j}l_{k+1}\left(\frac{N_k}{N_{j-1}}\right) ^{d+1}\,.
\end{align*}
Let $\textbf{i'}\in\partial_v\Lambda_{j-1}$. There exists $\textbf{l'}\in\partial_v\widetilde{\Lambda}_j$ such that $\textbf{i'}\in\partial_vC ^{(j-1)}(\textbf{l'})$ and so $\textbf{i'}$ is a good site at scale $j-1$ and the property $(ii)$ holds. 

\noindent We have $\overline{\Lambda}_{j}\subset \overline{\Lambda}_{j-1}$ and 
\begin{align}\label{eq:lambda1}
\dis(\overline{\Lambda}_j,\{0,x\})&\leq \min_{y\in\partial_v \overline{\Lambda}_{j-1}} \dis(\overline{\Lambda}_j, y)+\dis(y, \{0,x\})\,.
\end{align}
We claim that 
\begin{align}\label{eq:lambda1bis}
\forall y\in \partial_v \overline{\Lambda}_{j-1}\qquad\dis(\overline{\Lambda}_j, y)\leq N_{j}+1\,.
\end{align}
\begin{figure}[!ht]
\def\svgwidth{1\textwidth}
 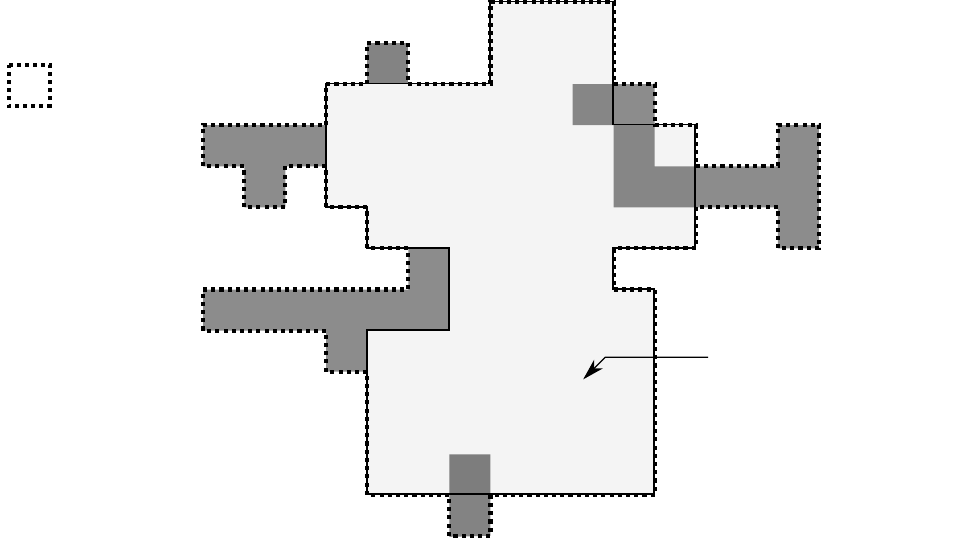
 \caption[fig4]{\label{fig4}Representation of $y,z$ and $\textbf{m}$, $\textbf{m'}$}
\end{figure}

 Let $y\in \partial_v \overline{\Lambda}_{j-1}$.
Let $z$ such that $z\in\overline{\Lambda}_{j-1}$ and $\|y-z\|_\infty=1$.
Let $\textbf{m}\in\sZ^d$ such that $z\in B_{N_{j-1}}(\textbf{m})$ (see figure \ref{fig4}). Since $y\in \partial_v \overline{\Lambda}_{j-1}$, there exists $\textbf{m'}\in\partial_v\widetilde{\Lambda}_j$ such that $\textbf{m}\in C^{(j-1)}(\textbf{m'})$. Thanks to property \ref{cond:2} of the induction hypothesis, we have $|C^{(j-1)}(\textbf{m'})|\leq l_j$. It follows that
\[\dis(\overline{\Lambda}_j, y)\leq \dis(\overline{\Lambda}_j, z)+\dis(z,y)\leq |C^{(j-1)}(\textbf{m'})|N_{j-1}+1\leq l_jN_{j-1}+1\leq N_j+1\,.\]
Combining inequalities \eqref{eq:lambda1} and \eqref{eq:lambda1bis}, we obtain
\begin{align}\label{eq:lambda2}
\dis(\overline{\Lambda}_j,\{0,x\})&\leq N_j+1+\dis(\partial_v \overline{\Lambda}_{j-1}, \{0,x\})\nonumber\\
&= N_j+1+\dis( \overline{\Lambda}_{j-1}, \{0,x\})-1\nonumber\\
&\leq N_j+\dis( \overline{\Lambda}_{j-1}, \{0,x\})\,.
\end{align}
\noindent Combining inequalities \eqref{eq:lambda2} and property \ref{cond:3} of the induction hypothesis, it follows that
\begin{align}\label{eqnum2}
\dis(\overline{\Lambda}_{j-1}, \{0,x\})\geq \dis(\overline{\Lambda}_{j}, \{0,x\})-N_j\geq  2N_M-2N_k -\sum_{l=j}^{k+1}N_l\,.
\end{align}
The property $(iii)$ follows and this concludes the induction.
We set finally $$\shell(e)=\partial_v\Lambda_1\,.$$ Thanks to property $(i)$, we get
\begin{align}\label{contshell}
|\shell(e)|\leq 2d 3^d(3d)^{k}l_{k+1}\left(\frac{\displaystyle N_k}{\displaystyle N_1}\right) ^{d+1}\,.
\end{align}
Since $\overline{\Lambda}_k\subset\overline{\Lambda}_1$, we deduce from inequality \eqref{eqnum:1} that
$$\dis\left(\{e\},\bigcup_{\textbf{i}\in\shell(e)}B_{N_1}(\textbf{i})\right)\geq (14\beta+2d)N_1\,.$$
Moreover, property $(iii)$ for $l=1$, together with the facts that $k\leq M-1$ and $l_M\geq M+2$ (since the sequence $(l_k)_{k\geq1}$ is increasing), imply that 
\begin{align*}
\dis(\overline{\Lambda}_1,\{0,x\})&\geq 2N_M-2N_{M-1}-(M-1)N_{M-1}-N_M\\&=(l_M-M-1)N_{M-1}\geq N_{M-1}\geq N_1\,.
\end{align*}
Therefore, the path $\gamma$ enters and exits $\partial_v\Lambda_1$ at least once. For $e\in\overline{\gamma}$, we denote by $k(e)\geq 2$ the largest integer $k$ such that  $\textbf{e}_3,\dots,\textbf{e}_k$ are $(p,q)$-bad at their respective scales (if $\textbf{e}_3$ is good, then we set $k(e)=2$). By construction, we have that $k(e)<M(\gamma)$. For $k\geq 3$, the number of edges  $e\in\overline{\gamma}$ such that $k(e)=k$ is at most $n_k(\gamma)|B_{N_k}\cap \E^d|=n_k(\gamma)dN_k^d$. The number of edges in $\overline{\gamma}$ such that $k(e)=2$ is at most $|\gamma|$. Finally, using inequality \eqref{contshell}, we have
\begin{align*}
\sum_{e\in\overline{\gamma}}|\shell(e)|&\leq \sum_{k=2}^{M-1}\big|\{e\in\overline{\gamma}:\,k(e)=k\}\big|\, 2d 3^d(3d)^{k}l_{k+1}\left(\frac{\displaystyle N_k}{\displaystyle N_1}\right) ^{d+1}\\
&\leq 2d3^d\left((3d)^2|\gamma|N_{3}N_2^d+\sum_{k=3}^{M-1} n_k(\gamma)N_{k+1}N_k^{2d}(3d)^kd \right)
\end{align*}
and similarly
\begin{align*}
\sum_{e\in\overline{\gamma}}|\shell(e)|^2&\leq \sum_{k=2}^{M-1}\big|\{e\in\overline{\gamma}:\,k(e)=k\}\big|  \left(2d 3^d(3d)^{k}l_{k+1}\left(\frac{\displaystyle N_k}{\displaystyle N_1}\right) ^{d+1}\right)^2\\
&\leq 4d^23^{2d}\left((3d)^4|\gamma|N_{3}^2N_2^{2d}+ \sum_{k=3}^{M-1}n_k(\gamma)N_{k+1}^2N_k^{3d}(3d)^{2k}d\right)\,. 
\end{align*}
This concludes the proof.
\end{proof}
Note that in the proof of proposition \ref{propconstr}, even if the site $\textbf{e}_1$ or $\textbf{e}_2$ is good, we do not bypass the edges of its box at that scale because we have to make sure that the bypass we build do not use the edge $e$. To avoid this problem, we build $\shell(e)$ in such a way that it is far enough from the edge $e$.

 Once the shells are built, we build the bypasses. Given a set $E$ of edges in $\gamma$ we would like to bypass, we control the total length of the bypasses with the help of the total size of the shells associated to the edges in $E$.

\begin{prop}[Construction of the bypasses]\label{propbypass} Let $p_c(d)<p<q$. Let $x\in\sZ^d$. Let us assume that $0,x\in\sC_q$. We consider a $q$-open path $\gamma$ between $0$ and $x$. Let $M$ be the smallest positive integer such that $n_M(\gamma)=0$ and let $\overline{\gamma}$ and $(\shell(e))_{e\in\overline{\gamma}}$ be defined as in proposition \ref{propconstr}. Let $E$ be a subset of $\overline{\gamma}$. There exists a path $\gamma'$ between $0$ and $x$ such that $\gamma'\cap E=\emptyset$, the edges in $\gamma'\setminus \gamma$ are $p$-open and $$|\gamma'\setminus\gamma| \leq 12\beta N_1\sum_{e\in E}|\shell(e)|\,.$$
\end{prop}

\noindent  We will need the following lemma to prove proposition \ref{propbypass}. To guarantee that we bypass an edge $e$ without using $e$ with the renormalization scheme, we build our bypass far enough from the edge $e$. The following lemma enables us to build a $p$-open path between two points in a $*$-connected set $\mathcal{I}$ of good boxes at scale $1$ in such a way that it avoids a given set of vertices $A$. When applying this lemma, we will take for the set $A$ the set of the endpoints of the edges we wish to bypass. The lemma provides a control on the length of the bypass depending on $|\mathcal{I}|$.
   
\begin{lem} \label{lem2}[Adaptation of lemma 3.2 in \cite{DembinRegularity}] Let $A$ be a subset of $\sZ^d$. Let $\mathcal{I}$ be a finite $*$-connected set such that $$\dis\Big(A,\bigcup_{\textbf{i}\in\mathcal{I}}B_{N_1}(\textbf{i})\Big)>(14\beta +2d)N_1\,.$$ Suppose that all the sites in $\mathcal{I}$ are $(p,q)$-good at scale $1$. Let $\textbf{j},\textbf{k}\in\cI$, $x\in B'_{N_1}(\textbf{j})$ be in the $p$-crossing cluster of $B_{N_1}(\textbf{j})$ and $y\in B'_{N_1}(\textbf{k})$ be in the $p$-crossing cluster of $B_{N_1}(\textbf{k})$. There exists a $p$-open path joining $x$ and $y$ of length at most $12\beta N_1|\mathcal{I}|$ that does not visit any point of $A$.
\end{lem}

\begin{proof}[Proof of lemma \ref{lem2}] 
Since $\mathcal{I}$ is a $*$-connected set of sites, there exists a self-avoiding $*$-connected path $(\textbf{i}_l)_{1\leq l\leq r}\subset \mathcal{I}$  such that $\textbf{i}_1=\textbf{j}$, $\textbf{i}_r=\textbf{k}$. Necessarily, we have $r\leq|\mathcal{I}|$. As all the sites in $\mathcal{I}$ are good at scale $1$, all the sites in $(\textbf{i}_l)_{1\leq l\leq r}$ are good at scale $1$. We define $x_1=x$ and $x_r=y$. For $l\in\{2,\dots,r-1\}$, we choose a point $x_l$ in the $p$-crossing cluster of the box $B_{N_1}(\textbf{i}_l)$. The point $x$ (respectively $y$) is at distance at most $2dN_1$ from $B_{N_1}(\textbf{j})$ (respectively $B_{N_1}(\textbf{k})$), therefore the points $x$ and $y$ are at distance at least $14\beta N_1$ from $A$. For $l\in\{2,\dots,r-1\}$, the point $x_l$ belongs to $B_{N_1}(\textbf{i}_l)$ and so it is at distance at least $14\beta N_1$ from the set $A$.
For $l\in\{1,\dots,r-1\}$, both points $x_l$ and $x_{l+1}$ belong to $B'_{N_1}(\textbf{i}_l)$. Using property $(iii)$ of a $p$-good box, we can build a $p$-open path $\gamma(l)$ from $x_l$ to $x_{l+1}$ of length at most $12\beta N_1$. As the points $x_l$ and $x_{l+1}$ are both at distance at least $14\beta N_1$ from $A$, the path $\gamma(l)$ does not go through a vertex in $A$. By concatenating the paths $\gamma(1),\dots,\gamma(r-1)$ in this order, we obtain a $p$-open path joining $x$ to $y$ of length at most $ 12 \beta N_1|\mathcal{\cI}|$ that does not visit any point in $A$.
\end{proof}

\begin{figure}[!ht]
\def\svgwidth{1\textwidth}
 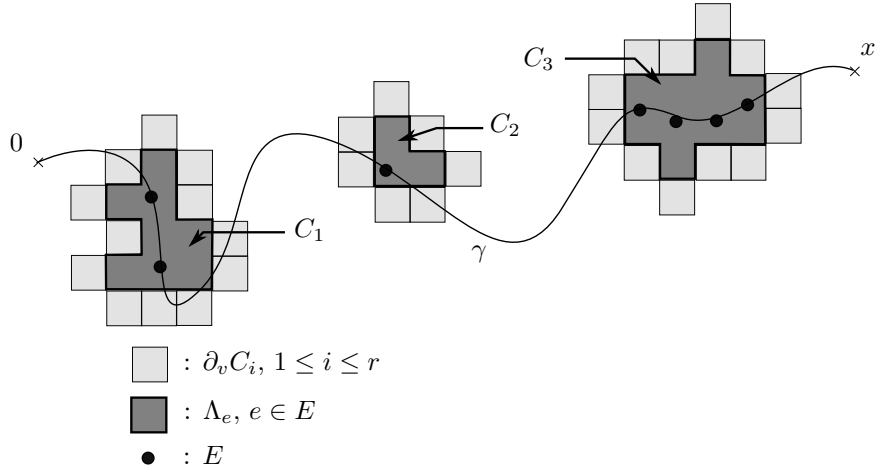
 \caption[fig2]{\label{fig2}Construction of $C_1,\dots,C_r$}
\end{figure}
\begin{proof}[Proof of proposition \ref{propbypass}] Let $(\shell(e))_{e\in\overline{\gamma}}$ be the family of the shells built in proposition \ref{propconstr}. For any $e\in E$, we denote $$\Lambda_e=\inte(\shell(e))\,,\quad \overline{\Lambda}_e=\bigcup_{\textbf{i}\in\Lambda_e}B_{N_1}(\textbf{i})\,.$$ We write 
$$\bigcup_{e\in E}\Lambda_e=\bigcup_{k=1}^r C_k$$ 
where $C_1,\dots,C_r$ are disjoint connected components of sites, ordered in such a way that $C_1$ is the first component visited by $\gamma$, $C_2$ is the second and so on (see Figure \ref{fig2}). We set, for $k\in\{1,\dots,r\}$, $$\overline{C}_k=\bigcup_{\textbf{i}\in C_k}B_{N_1}(\textbf{i})\,,\quad\overline{\partial_vC_k}=\bigcup_{\textbf{i}\in \partial_v C_k}B_{N_1}(\textbf{i})\,.$$ 
Thanks to the second inequality of proposition \ref{propconstr}, for every $e\in E$, we have $\dis(\{0,x\},\overline{\Lambda}_e)\geq N_1$ and thus, for every $k\in\{1,\dots,r\}$, we have also $\dis(\{0,x\},\overline{C}_k)\geq N_1$. This implies that $\gamma$ enters and exits each connected component $C_k$ at least once.
Let us introduce some further notations (see Figure \ref{fig3}). We write $\gamma=(x_0,\dots,x_n)\,.$ We define $$\tau_{in}(1)=\min\big\{\,j\geq 1:\, x_j\in \overline{\partial_v C_1}\,\big\}\,,$$ $$\tau_{out}(1)=\max\big\{\,j\geq \tau_{in}(1) :\, x_j\in \overline{\partial_v C_1}\,\big\}\,.$$ 
Let  $k\in\{1,\dots,r\}$. Suppose that $\tau_{in}(1),\dots,\tau_{in}(k)$ and $\tau_{out}(1),\dots,\tau_{out}(k)$ are defined. We define then $$\tau_{in}(k+1)=\min\big\{\,j\geq \tau_{out}(k):\, x_j\in \overline{\partial_v C_{k+1}}\,\big\}\,,$$ $$\tau_{out}(k+1)=\max\big\{\,j\geq \tau_{in}(k+1) :\, x_j\in \overline{\partial_v C_{k+1}}\,\big\}\,.$$ Let us fix $k\in\{1,\dots,r\}$. Let $B_{in}(k)$ be the $N_1$-box in $\partial_v C_k$ containing $x_{\tau_{in}(k)}$, $B_{out}(k)$ be the $N_1$-box in $\partial_v C_k$ containing $x_{\tau_{out}(k)}$. Since $B_{in}(k)$ (respectively $B_{out}(k)$) is a good box, it contains a unique crossing cluster $\sC_{in}$ (respectively $\sC_{out}$). 
 Moreover, we have $|\gamma\cap B'_{in}(k)|\geq N_1$ (respectively $|\gamma\cap B'_{out}(k)|\geq N_1$), thus by property $(iv)$ of a good box, there exists a vertex $y_{in}(k)$ in $\gamma\cap B'_{in}(k)\cap \sC_{in}$ (respectively $y_{out}(k)$ in $\gamma\cap B'_{out}(k)\cap \sC_{out}$). We select such a vertex according to some deterministic rule.
\begin{figure}[!ht]
\def\svgwidth{1\textwidth}
 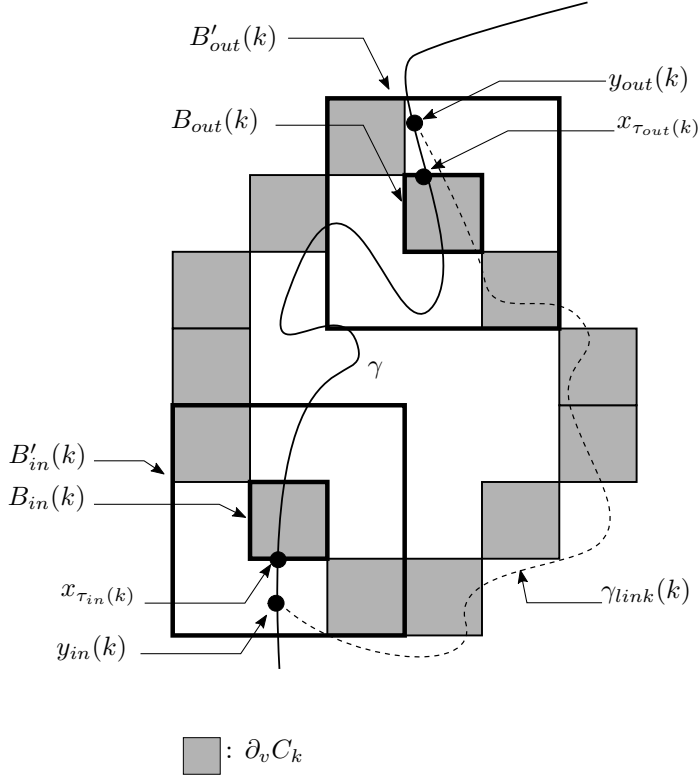
 \caption[fig3]{\label{fig3}Construction of $\gamma_{link}(k),\,1\leq k \leq r$}
\end{figure}
Thanks to the first inequality of proposition \ref{propconstr}, we have
$$\forall e\in E\qquad\dis(\{e\},\overline{\partial_v C_k})\geq \dis(\{e\},\partial_v\overline{\Lambda_e})\geq (14\beta+2d) N_1\,,$$
 whence $$ \dis(E,\overline{\partial_vC_k})\geq (14\beta+2d) N_1\,.$$
We apply lemma \ref{lem2} by taking the extremities of the edges of $E$ for the set $A$ and $\partial_v C_k$ for the set $\cI$: there exists a $p$-open path $\gamma_{link}(k)$ joining $y_{in}(k)$ and $y_{out}(k)$ of length at most $12\beta N_1 |\partial_v C_k|$ which does not visit any edge in $E$. For $k$ in $\{1,\dots,r-1\}$, we denote by $\gamma(k)$ the portion of $\gamma$ between $y_{out}(k)$ and $y_{in}(k+1)$. Let $\gamma(0)$ (respectively $\gamma(r)$) be the portion of $\gamma$ from $y$ to $y_{in}(1)$  (respectively from $y_{out}(r)$ to $z$).  We obtain a path $\gamma'$ joining $y$ and $z$ by concatenating the paths $\gamma(0),\gamma_{link}(1),\gamma(1),\dots,\gamma_{link}(r),$ $\gamma(r)$ in this order. We can extract from $\gamma'$ a self-avoiding path $\gamma''$. By construction, the edges in $\gamma''\setminus \gamma$ are $p$-open. Let us estimate the number of edges in $\gamma''\setminus \gamma$. Since $$\gamma''\setminus\gamma \subset \bigcup_{k=1}^{r} \gamma_{link}(k)\,,$$ then
 \begin{align*}
 |\gamma'' \setminus\gamma| &\leq \sum_{k=1}^{r} |\gamma_{link}(k)| \leq \sum_{k=1}^{r} 12\beta N_1 |\partial_ v C_k| \leq 12 \beta N_1 \sum_{e\in E}|\shell(e)| \, .
 \end{align*}
This yields the desired result.\end{proof}
\section{Control of the probability of being a bad box}\label{goodbox}

In this section, we prove that the probability of being a $(p,q)$-bad $N_1$-box at scale $1$ decays exponentially fast with $N_1$. The main difficulty is to get an exponential decay which is uniform in $p$. For that purpose, we introduce a parameter $p_0>p_c(d)$ and we obtain an exponential decay which is uniform for all $p\geq p_0$. We recall that $(p,q)$-good boxes at scale $1$ were defined in definition \ref{defgoodbox}. 
\begin{thm}\label{thm5} Let $p_0>p_c(d)$. There exist positive constants $\beta(p_0)$, $N_0(p_0)$, $\delta_0(p_0)$ and $C(p_0)$ such that 
\begin{align*}
&\forall p\geq p_0\quad \forall N\geq N_0\quad \forall q\in [p,p+\delta_0]\\
&\hspace{3cm}\Prb(B_{N}\text{ is $(p,q)$-bad at scale $1$})\leq \exp(-C(p_0)N)\hfill\,.
\end{align*}
\end{thm}
\begin{proof} We prove here only the exponential decay for the property $(iv)$, as the exponential decay for the other properties $(i), (ii),(iii)$ were already proven in \cite{DembinRegularity}. 
We refer also to the proof of lemma 3.5 in \cite{GaretMarchandProcacciaTheret}. For given parameters $p,q$ satisfying $p_c(d)<p_0\leq p\leq q\leq 1$, we denote by $\Prb_{p,q}$ the probability associated to two coupled Bernoulli percolations of parameters $p,q$. As usual, the coupling is such that the edges are independent and every $p$-open edge is also $q$-open. We define $A_N$ as the event that there exists a $p$-crossing cluster $\sC$ in $B'_N$ and a $q$-open path $\gamma \subset B'_N$ such that $|\gamma|=N$ and $\gamma$ does not intersect $\sC$. The following inequality  was proven in \cite{GaretMarchandProcacciaTheret}:
\begin{align*}
\Prb_{p,q}(A_N)\leq \Prb_{p,p}(A_N)\exp\left(N\log\left(1+\frac{q-p}{p}\right)\right)\,.
\end{align*} 
The quantity $\Prb_{p,p}(A_N)$ decays exponentially with $N$, as it satisfies property $(ii)$ but not $(i)$ (see lemma 3.5. in \cite{GaretMarchandProcacciaTheret}). Thus there exist positive constants $\kappa_1(p_0)$ and $\kappa_2(p_0)$ depending on $p_0$ and $N_0$ such that 
$$\forall N\geq 1\qquad\Prb_{p,p}(A_N)\leq \kappa_1(p_0)\exp(-\kappa_2(p_0)N)\,.$$
Now there exists a constant $\delta_0>0$ depending only on $p_0$ such that, if the parameters $q$ and $p$ are such that $p_0\leq p\leq q$ and $q-p\leq \delta_0$, then
$$\kappa_2(p_0)>\log\left(1+\frac{\delta_0}{p_0}\right)\geq \log\left(1+\frac{q-p}{p_0}\right)\,,$$  and so
$$\Prb_{p,q}(\text{$B_{N}$ is a $(p,q)$-bad box})\leq A(p_0)\exp(-B(p_0)N)\,,$$
with $A(p_0)=\kappa_1(p_0)$ and $B(p_0)=\kappa_2(p_0)-\log (1+(q-p)/p_0)\,.$
The result follows.
\end{proof}
We prove in the following theorem that it is possible to tune the scales $l_k$, $k\geq 1$, in such a way that the probability of being a bad site at scale $k\geq 1$ decays at least exponentially fast with $N_k$. We recall that good sites at scale $k\geq 2$ were defined in definition \ref{defgoodbigbox}.
\begin{thm}\label{thm6} Let $p_0>p_c(d)$. Once $p_0$ is fixed, we choose $\beta(p_0)>0$ given by theorem \ref{thm5}. There exist positive constants $\beta(p_0)$, $\delta_0(p_0)$, $C(p_0)$, $R$ and $l_0$ such that, for any non-decreasing sequence of scales $(l_k)_{k\geq 1}$ satisfying $l_1\geq l_0$, we have
\begin{align*}
\forall p\geq p_0&\quad \forall q\in [p,p+\delta_0]\quad\forall k\geq 1\\
&\Prb(\textbf{0}\text{ is $(p,q)$-bad at scale $k$})\leq \exp\left(-C(p_0)\frac{N_k}{(2R)^{d(k-1)}}\right)\,.
\end{align*}
\end{thm}
\begin{proof}
Let $p_0>p_c(d)$. We start with scale $1$. Let $N_0(p_0),\delta_0(p_0),B(p_0)$ associated to $p_0$ as in theorem \ref{thm5}. Let $R$ be the smallest integer larger than $52\beta$. We choose $l_1\geq\max( N_0, (2R)^d)$ large enough to ensure that
\begin{align}\label{eqcond1}
 d\log7\leq C(p_0)\frac{l_1}{4R^d}
\end{align}
and
\begin{align}\label{eqetoile}
\forall l\geq l_1\qquad\log2+d\log(3l)\leq d(\log 7 ) l\,.
\end{align}
For $\textbf{i}\in\sZ^d$, the event $\{\textbf{i}\text{ is good at scale $1$}\}$ depends only on the edges in $B_{26\beta N_1}(\textbf{i})$. Thus, for $\textbf{i},\textbf{j}\in\sZ^d$, if $\|\textbf{i}-\textbf{j}\|_\infty\geq R\geq 52\beta$, the events $$\{\textbf{i}\text{ is good at scale $1$}\}\qquad\text{and}\qquad\{\textbf{j}\text{ is good at scale $1$}\}$$ are independent. 
Thanks to theorem \ref{thm5}, there exist positive constants $C(p_0)$, $N_0(p_0)$ and $\delta_0(p_0)$ such that 
$$\forall q\in [p,p+\delta_0]\quad \forall N\geq N_0\qquad\Prb(B_{N}\text{ is $(p,q)$-bad})\leq \exp(-C(p_0)N)\,.$$
Let now $k\geq 1$ be fixed and suppose that the first $k$ scales $l_1,\dots,l_k$ have been chosen, and that the following inequality holds:
\begin{align}\label{propher}
\Prb(\textbf{0}\text{ is $(p,q)$-bad at scale $k$})\leq \exp\left(-C(p_0)\frac{N_k}{(2R)^{d(k-1)}}\right)\,,
\end{align}
and that two sites $\textbf{i},\textbf{j}$ at scale $k$ are independent whenever $\|\textbf{i}-\textbf{j}\|_\infty\geq R$.
For the scale $k+1$, we have 
\begin{align*}
\Prb&(\textbf{0}\text{ is $(p,q)$-bad at scale $k+1$})\\
&\hspace{1,5cm}=\Prb\left(\exists\textbf{i}\in\underline{B}'_{k+1}(\textbf{0}) : \,|C ^{(k)}(\textbf{i})|> l_{k+1} \right)\\
&\hspace{1,5cm}\leq \sum _{\textbf{i}\in\underline{B}'_{k+1}(\textbf{0})}\sum_{m> l_{k+1}}\sum_{\Gamma\in\animals(m)}\Prb\left( \forall \textbf{j}\in\textbf{i}+\Gamma\quad \textbf{j}\text{ is bad at scale $k$}\right)\,,
\end{align*}
where $\animals(m)$ is the set of $*$-connected sets of cardinality $m$ containing the site $\textbf{0}$.  We have the following bound (see for instance Grimmett \cite{Grimmett99}, p85): $$|\animals(m)|\leq7^{dm}\,.$$ 
Using the translation invariance of the model, we obtain 
\begin{align*}
\Prb&(\textbf{0}\text{ is $(p,q)$-bad at scale $k+1$})\\
&\hspace{2cm}\leq |\underline{B}'_{k+1}(\textbf{0})|\sum_{m> l_{k+1}}\sum_{\Gamma\in\animals(m)}\Prb\left( \forall \textbf{j}\in\Gamma\quad \textbf{j}\text{ is bad at scale $k$}\,\right)\,.
\end{align*}
To bound from above the previous probability, we will use the fact that distant sites at scale $k$ are independent. The following lemma allows us to extract from $\Gamma$ a subset $\Gamma_0$ in which sites are at mutual distances larger or equal than $R$.

\begin{lem}\label{extractlem}Let $\Gamma$ be a finite subset of $\sZ^d$. There exists a subset $\Gamma_0$ of $ \Gamma$ such that $|\Gamma_0|\geq |\Gamma|/R^d$ and
$$\forall\,\textbf{i},\textbf{j}\in\Gamma_0\qquad \textbf{i}\neq \textbf{j}\implies\|\textbf{i}-\textbf{j}\|_\infty \geq R\,.$$
\end{lem}
\begin{proof}The following collection of sets forms a partition of $\sZ^d$:
$$i+R\sZ^d,\quad i\in B_{R}\,.$$
Therefore the set $\Gamma$ can be partitioned as follows:
$$\Gamma=\bigcup_{i\in B_{R}}\left((i+R\sZ^d)\cap\Gamma\right)\,.$$
Since $|B_R|=R^d$, then there exists $i_0$ such that the set $\Gamma_0=(i_0+R\sZ^d)\cap\Gamma$ satisfies the conditions stated in  the lemma.
\end{proof}
\noindent 
Let $\Gamma$ be a fixed lattice animal. Let $\Gamma_0$ be a subset of $\Gamma$ that satisfies the conditions stated in lemma \ref{extractlem}. Let $\textbf{i}\neq \textbf{j}\in\Gamma_0$. By the induction hypothesis, the events $\{\textbf{i}\text{ is good at scale $k$}\}$ and
$\{\textbf{j}\text{ is good at scale $k$}\}$ are independent.  
Using the inequality of the induction hypothesis, we have thus
\begin{align*}
\Prb&(\textbf{0}\text{ is $(p,q)$-bad at scale $k+1$})\\
&\hspace{1,5cm}\leq(3l_{k+1})^d\sum_{m> l_{k+1}}\sum_{\Gamma\in\animals(m)}\Prb\left( \forall \textbf{j}\in\Gamma\quad \textbf{j}\text{ is bad at scale $k$}\,\right)\\
&\hspace{1,5cm}\leq (3l_{k+1})^d\sum_{m> l_{k+1}}\sum_{\Gamma\in\animals(m)}\Prb\left( \forall \textbf{j}\in\Gamma_0\quad \textbf{j}\text{ is bad at scale $k$}\,\right)\\
&\hspace{1,5cm}\leq (3l_{k+1})^d\sum_{m> l_{k+1}}\sum_{\Gamma\in\animals(m) }\exp\left(-C(p_0)\frac{N_k}{(2R)^{d(k-1)} }\frac{m}{R^d}\right)\\
&\hspace{1,5cm} \leq (3l_{k+1})^d\sum_{m> l_{k+1}}7^{dm}\exp\left(-C(p_0)\frac{N_k}{(2R)^{d(k-1)}} \frac{m}{R^d}\right)\,.
\end{align*}
Since $l_1\geq (2R) ^d$ and $l_1\leq \cdots\leq l_k$, then $N_k\geq (2R) ^{dk}$. Using \eqref{eqcond1}, we see that 
\begin{align}\label{eqn_1}
\forall m\geq 1\qquad 7^{dm}\exp\left(-C(p_0)\frac{N_k}{4(2R)^{d(k-1)}} \frac{m}{R^d}\right)\leq 1 
\end{align}
and 
$$\exp\left(-C(p_0)\frac{3N_k}{4(2R)^{d(k-1)}R^d}\right)\leq\frac{1}{7^{3d}}\leq \frac{1}{2}\,.$$
It follows that
\begin{align}\label{eqavantder}
\Prb&(\textbf{0}\text{ is $(p,q)$-bad at scale $k+1$})\nonumber\\
&\hspace{3cm}\leq (3l_{k+1})^d\sum_{m> l_{k+1}}\exp\left(-C(p_0)\frac{3N_k}{4(2R)^{d(k-1)}} \frac{m}{R^d}\right)\nonumber\\
&\hspace{3cm}\leq 2(3l_{k+1})^d\exp\left(-C(p_0)\frac{3N_k}{4(2R)^{d(k-1)}} \frac{l_{k+1}}{R^d}\right)\,.
\end{align}
The condition \eqref{eqetoile} on $l_1$ and the fact that $l_1\leq \cdots\leq l_{k+1}$ imply that 
\begin{align}\label{eqn_2}
\log2+d\log(3l_{k+1})\leq d(\log 7)l_{k+1}\,.
\end{align}
Thanks to inequalities \eqref{eqn_1} and \eqref{eqn_2}, we obtain
\begin{align}\label{eqavantder2}
2(3l_{k+1})^d\exp\left(-C(p_0)\frac{N_{k+1}}{4(2R)^{d(k-1)}R^d}\right)\leq 1\,.
\end{align}
Finally, combining inequalities \eqref{eqavantder} and \eqref{eqavantder2}, we obtain
\begin{align*}
\Prb(\textbf{0}\text{ is $(p,q)$-bad at scale $k+1$})\leq\exp\left(-C(p_0) \frac{N_{k+1}}{(2R)^{dk}} \right)\,.
\end{align*}
This yields the inequality at rank $k+1$.

For $\textbf{i}\in\sZ^d$, the event $\{\textbf{i}\text{ is good at scale $k+1$}\}$ depends on the states of the sites in
$$\mathfrak{B}_{k+1}(\textbf{i})=\big\{\textbf{l}\in\sZ^d:\,\|\textbf{l}-\textbf{j}\|_{\infty}\leq 2 l_{k+1},\,\textbf{j}\in\underline{B}'_{k+1}(\textbf{i})\,\big\}\,.$$
Indeed, for $\textbf{j}\in\underline{B}'_{k+1}(\textbf{i})$, in order to determine whether $|C^{(k)}(\textbf{j})|\leq l_{k+1}$ or not, we only need to reveal the states of the sites in $\mathfrak{B}_{k+1}(\textbf{i})$. By the induction hypothesis, two sites $\textbf{l},\textbf{m}$ at scale $k$ are independent whenever $\|\textbf{l}-\textbf{m}\|_\infty\geq R$. Now, if $\|\textbf{i}-\textbf{j}\|_\infty\geq R$, then 
$$\forall \textbf{l}\in\mathfrak{B}_{k+1}(\textbf{i})\quad \forall \textbf{m}\in \mathfrak{B}_{k+1}(\textbf{j}) \qquad \|\textbf{l}-\textbf{m}\|_\infty\geq Rl_{k+1}-8l_{k+1}\geq R$$
and the events $\{\textbf{i}\text{ is good at scale $k+1$}\}$ and
$\{\textbf{j}\text{ is good at scale $k+1$}\}$ are independent. This concludes the induction.
\end{proof}

\noindent For $k\geq 1$, we define the trace $\gamma^{(k)}$ of $\gamma$ on the lattice at scale $k$ as
$$\gamma ^{(k)}=\left\{\,\textbf{i}\in\sZ^d:\,B_{N_k}(\textbf{i})\cap\gamma\neq\emptyset\,\right\}\,.$$
The following lemma gives a deterministic control on the length of this trace.
\begin{lem}\label{probest}[lemma 3.4 in \cite{GaretMarchandProcacciaTheret}] For any path $\gamma$ in $\sZ^d$, we have $$\forall k\geq 1 \qquad|\gamma^{(k)}|\leq 3^d\left(1+\frac{|\gamma|+1}{N_k}\right)\, .$$
\end{lem}

\noindent The following proposition shows that we can choose adequately the sequence $(N_k)_{k\geq 1}$ in order to control the quantity that appears in proposition \ref{propconstr}.
\begin{prop}\label{propcontmauvais}
 Let $p_0>p_c(d)$. Once $p_0$ is fixed, we choose $\beta(p_0)$ given by theorem \ref{thm5}. There exist positive constants  $l_0(p_0)$, $n_0$, $\delta_0(p_0)$, $A_1(p_0)$ such that, if we define the sequence of scales $(l_k)_{k\geq 1}$ by setting $N_1=l_0$ and
$$\forall k \geq 1 \qquad N_{k+1}=N_k^{2d} \,$$ 
and we define $M=M(\gamma)$ as the smallest integer such that $n_M(\gamma)=0$, 
then we have
\begin{align*}
&\forall p\geq p_0\quad \forall q\in [p,p+\delta_0]\quad\forall n\geq n_0\\
&\Prb\left(\begin{array}{c}\text{There exists a path $\gamma$ starting }\\\text{from $0$ such that $|\gamma|\leq n$ and}\\\sum_{k=3}^M n_k(\gamma) N_{k+1}^2N_{k}^{3d}(3d)^{2k}\geq n \\ \text{or }N_{M(\gamma)}>n ^{1/3d}\end{array}\right)\leq \exp\left(-A_1(p_0)n^\frac{1}{6d^2+1}\right)\,.
\end{align*}
\end{prop}

\begin{proof}[Proof of proposition \ref{propcontmauvais}]
Let $p_0>p_c(d)$. Let $\delta_0(p_0)$, $C(p_0)$, $l_0$ be the constants given by theorem \ref{thm6}. We can assume that $l_0\geq 2(3d)^2$ (if this is not the case, we replace $l_0$ by $\max(l_0,2(3d)^2)$). Let $p>p_0$ and $q\in[p,p+\delta_0]$. Let $(N_k)_{k\geq 1}$ be the sequence defined in the statement of the proposition. We have $l_1\geq l_0$ and $l_1\geq 2(3d)^2$.  We set $$\forall k\geq 1\qquad\delta_k=\frac{1}{N_{k+1}^2N_k^{3d+1}}\,.$$  We have then
$$\sum_{k=3}^{M(\gamma)-1}(3d) ^{2k}N_{k+1}^2N_k^{3d}\delta_k=\sum_{k=3}^{M(\gamma)-1} \frac{(3d)^{2k}}{N_k}\leq \sum_{k=3}^\infty \left(\frac{(3d)^2}{l_1}\right)^k\leq\sum_{k=3}^\infty \left(\frac{1}{2}\right)^k<1\,.$$ We choose $n_0$ large enough so that $N_1\leq n_0^{1/3d}$. Let $n\geq n_0$. Let $\overline{M}$ be the largest integer such that $N_{\overline{M}}\leq n ^{1/3d}$. Therefore $N_{\overline{M}+1}=(N_{\overline{M}}) ^{2d}\leq n ^{2/3}<n$. We have
\begin{align}\label{jsap}
\Prb&\left(\begin{array}{c}\text{There exists a path $\gamma$ starting from $0$ such that $|\gamma|\leq n$}\\\text{and $\sum_{k=3}^{M(\gamma)-1} n_k(\gamma) (3d) ^{2k}N_{k+1}^2N_k^{3d}\geq n$ or $N_{M(\gamma)}> n^{1/3d}$}\end{array}\right)\nonumber\\
&\leq \Prb\left(\begin{array}{c}\text{There exists a path $\gamma$ starting from $0$ such that $|\gamma|\leq n$ and}\\\text{$\sum_{k=3}^{M(\gamma)-1} n_k(\gamma) (3d) ^{2k}N_{k+1}^2N_k^{3d}\geq \sum_{k=3}^{M(\gamma)-1} (3d) ^{2k}N_{k+1}^2N_k^{3d}\delta_k n$} \\\text{or $M(\gamma)>\overline{M}$}\end{array}\right)\nonumber\\
&\leq \Prb\left(\begin{array}{c}\text{There exists a path $\gamma$ starting from $0$ such that $|\gamma|\leq n$ and}\\\text{$\sum_{k=3}^{M(\gamma)-1}n_k(\gamma) (3d) ^{2k}N_{k+1}^2N_k^{3d}\geq \sum_{k=3}^{M(\gamma)-1}(3d) ^{2k}N_{k+1}^2N_k^{3d}\delta_k  n$}\\\text{and $M(\gamma)\leq \overline{M}$}\end{array}\right)\nonumber\\
&\hspace{1cm}+\Prb\left(\begin{array}{c}\text{There exists a path $\gamma$ starting from $0$ }\\\text{ such that $|\gamma|\leq n$ and $M(\gamma)>\overline{M}$}\end{array}\right)\nonumber\\
&\leq \sum_{k=3}^{\overline{M}-1}\Prb\left(\begin{array}{c}\text{There exists a path $\gamma$ starting from $0$ }\\\text{such that $|\gamma|\leq n$ and $n_k(\gamma)\geq \delta_k n$}\end{array}\right)\nonumber\\
&\hspace{1cm}+\Prb\left(\begin{array}{c}\text{There exists a path $\gamma$ starting from $0$ }\\\text{ such that $|\gamma|\leq n$ and $n_{\overline{M}}(\gamma)\neq 0$}\end{array}\right)\,.
\end{align}
Let $k\in\{3,\dots,M(\gamma)-1\}$. To control the probability that there exists a path $\gamma$ such that $n_k(\gamma)$ exceeds $\delta_k n$,  a natural strategy would be to sum over all the possible traces $\gamma^{(k)}$ of $\gamma$ at scale $k$. However, this would lead to a combinatorial term that grows too fast with $n$. Instead of summing over all possible traces $\gamma^{(k)}$ at scale $k$, we sum over all possible traces $\gamma^{(k+2)}$ at scale $k+2$ (summing over all possible traces $\gamma^{(k+1)}$ at scale $k+1$ also leads to a too large combinatorial term). Let $\Gamma$ be a set of sites at scale $k+2$, we denote by $\underline{\Gamma}_{(k)}$ the set of the sites at scale $k$ which are contained in $\Gamma$, \textit{i.e.},
$$\underline{\Gamma}_{(k)}=\bigcup_{\textbf{i}\in\Gamma}\,\bigcup_{\textbf{j}\in \underline{B}_{k+2}(\textbf{i})}\underline{B}_{k+1}(\textbf{j})\,.$$
Let $\gamma$ be a path starting from $0$ such that $|\gamma|\leq n$. With the previous notation, we have that $$\gamma^{(k)}\subset\underline{\gamma^{(k+2)}}_{(k)}\,.$$ 
We can therefore bound from above the number of bad sites $n_k(\gamma)$ at scale $k$ that $\gamma$ crosses by the number of bad sites at scale $k$ contained in\smash{ $\underline{\gamma^{(k+2)}}_{(k)}$}. We denote this number by $n_k^{k+2}(\gamma^{(k+2)})$, namely,
$$n_k^{k+2}(\gamma^{(k+2)})=\left|\left\{\, \textbf{i}\in \underline{\gamma^{(k+2)}}_{(k)}: \textbf{i}\text{ is bad at scale $k$}\right\}\right|\,.$$
We denote by $\lfloor x\rfloor$ the greatest integer smaller than or equal to the real number $x$.
Using lemma \ref{probest} together with the fact that $N_{k+2}<n$, we obtain
\begin{align*}
|\gamma ^{(k+2)}|\leq 3^d\left(1+\frac{|\gamma|+1}{N_{k+2}}\right)\leq 3^d\left(1+\frac{2n}{N_{k+2}}\right)\leq  3^{d+1}\frac{n}{N_{k+2}}\,.
\end{align*}
We have thus
\begin{align}\label{eq1}
\Prb&\left(\begin{array}{c}\text{There exists a path $\gamma$ starting from $0$ }\\\text{such that $|\gamma|\leq n$ and $n_k(\gamma)\geq \delta_k n$}\end{array}\right)\nonumber \\
&\leq \,\Prb\left(  \begin{array}{c}\text{There exists a path $\gamma$ starting from $0$}\\ \text{such that }|\gamma|\leq n,\,n_k^{k+2}(\gamma^{(k+2)})\geq \delta_k n\end{array}\right)\nonumber\\
&\leq \,\Prb\Bigg(\bigcup_{\Gamma\in\animals\big(\big\lfloor \frac{3 ^{d+1}n}{N_{k+2}}\big\rfloor\big)}\left\{\begin{array}{c}\text{There exists a path $\gamma$ starting from $0$ such}\\ \text{that }|\gamma|\leq n,\,n_k^{k+2}(\Gamma)\geq \delta_k n,\,\gamma^{(k+2)}\subset \Gamma \end{array}\right\}\Bigg)\nonumber\\
&\leq \sum_{\Gamma\in\animals\big(\big\lfloor \frac{3 ^{d+1}n}{N_{k+2}}\big\rfloor\big)}\Prb\left( \begin{array}{c}\text{There exists a path $\gamma$ starting from $0$ such}\\\text{that $|\gamma|\leq n$, $n_k( \underline{\Gamma}_{(k)}))\geq \delta_k n$ and  } \gamma^{(k+2)}\subset \Gamma\end{array} \right)\nonumber\\
&\leq \sum_{\Gamma\in\animals\big(\big\lfloor \frac{3 ^{d+1}n}{N_{k+2}}\big\rfloor\big)}\Prb\big(  n_k( \underline{\Gamma}_{(k)})\geq \delta_k n\big)\,.
\end{align}
Let $\Gamma$ be a fixed $*$-connected set of sites at scale $k+2$ containing $0$ and of size $\lfloor3^{d+1} n/N_{k+2}\rfloor$. We have
$$|\underline{\Gamma}_{(k)}|\leq (l_{k+2}l_{k+1})^d |\Gamma|\leq (l_{k+2}l_{k+1})^d\left\lfloor \frac{3^{d+1}n}{N_{k+2}}\right\rfloor\,.$$
 On the event $\{n_k( \underline{\Gamma}_{(k)})\geq \delta_k n\}$, using the same arguments as in lemma \ref{extractlem}, we can extract a subset $\widetilde{\Gamma}$ of $\underline{\Gamma}_{(k)}$ such that $n_k(\widetilde\Gamma)\geq \delta_k n/R^d$ and $$ \forall\,\textbf{i},\textbf{j}\in\widetilde{\Gamma}\qquad \textbf{i}\neq \textbf{j}\implies\|\textbf{i}-\textbf{j}\|_\infty\geq R \geq 3\,.$$ This implies that the events $\{\textbf{i}\text{ is good at scale $k$}\}$ and
$\{\textbf{j}\text{ is good at scale $k$}\}$ are independent for $\textbf{i}\neq \textbf{j}\in \widetilde{\Gamma}$. From the proof of lemma \ref{extractlem} and the way $\widetilde{\Gamma}$ is constructed, we see that there are at most $R^d$ choices for the set $\widetilde{\Gamma}$. Let us set $$\pp=\Prb(\textbf{0}\text{ is bad at scale $k$})\,.$$ Let $(X_i)_{i\geq 1}$ be a sequence of i.i.d. Bernoulli random variables with parameter~$\pp$.
From the previous discussion, we have
\begin{align}\label{eqcomp}
\Prb&\left(  n_k( \underline{\Gamma}_{(k)})\geq \delta_k n\right)\nonumber\\
&\hspace{1cm}\leq \,\Prb\left(  n_k( \widetilde{\Gamma})\geq \frac{\delta_kn}{R^d}\right)\nonumber\\
&\hspace{1cm}\leq \, R^d\,\Prb\left(\sum_{i=1}^{(l_{k+2}l_{k+1})^d\left\lfloor \frac{3^{d+1}n}{N_{k+2}}\right\rfloor}X_i\geq \frac{\delta_kn}{R^d}\right)\nonumber\\
&\hspace{1cm}\leq\, R^d\, \Prb\left(\sum_{i=1}^{(l_{k+2}l_{k+1})^d\left\lfloor \frac{3^{d+1}n}{N_{k+2}}\right\rfloor}X_i\geq \frac{\delta_k N_{k+2}}{3(3R l_{k+1}l_{k+2})^d  }(l_{k+2}l_{k+1})^d\left\lfloor \frac{3^{d+1}n}{N_{k+2}}\right\rfloor  \right)\,.
\end{align}
Let us set $$\delta'= \frac{\delta_k N_{k+2}}{3(3R l_{k+1}l_{k+2})^d}=\frac{ 1}{3(3R)^d N_{k+2}^{d-1} N_{k+1}^2  N_k^{2d+1}}<\frac{1}{2}\,.$$
Using theorem \ref{thm6}, we have
$$\frac{\delta'}{\pp}\geq \frac{ 1}{3(3R)^d N_{k+2}^{d-1} N_{k+1}^2  N_k^{2d+1}}\exp\left(\frac{C(p_0)N_k}{(2R)^{d(k-1)}}\right)\,.$$
Since $N_{k+1}=N_k^{2d}$ , we have
\begin{align*}
\frac{\delta'}{\pp}&\geq \frac{ 1}{3(3R)^d N_{k}^{4d^2(d-1)+6d+1}}\exp\left(\frac{C(p_0)N_k}{(2R)^{d(k-1)}}\right)\\
&=\frac{ 1}{3(3R)^d l_0^{(2d)^k (4d^2(d-1)+6d+1)}}\exp\left(\frac{C(p_0)l_0^{(2d)^k}}{(2R)^{d(k-1)}}\right)\,.
\end{align*}
We can choose $l_0$ large enough depending on $d$, $\beta$ and $p_0$ such that
$$\forall k \geq 1\qquad\frac{\delta'}{\pp}\geq  8\,.$$
As $\delta'>\pp$, we can use the Cramér-Chernoff inequality, we obtain
\begin{align*}
& \Prb\left(\sum_{i=1}^{(l_{k+2}l_{k+1})^d\left\lfloor \frac{3^{d+1}n}{N_{k+2}}\right\rfloor}X_i\geq \delta'(l_{k+2}l_{k+1})^d\left\lfloor \frac{3^{d+1}n}{N_{k+2}}\right\rfloor\right)\\
 &\hspace{1cm}\leq \exp\left(-(l_{k+2}l_{k+1})^d\left\lfloor \frac{3^{d+1}n}{N_{k+2}}\right\rfloor\left(\delta'\log\frac{\delta'}{\pp}+(1-\delta')\log\frac{1-\delta'}{1-\pp}\right)\right)\,.
 \end{align*}
Using the convexity of the function $x\mapsto -\log(1-x)$, we have
 $$\forall x\in[0,1/2]\qquad-\log(1-x)\leq 2(\log2) x\,.$$
Since $\delta'<1/2$, we have
$$-(1-\delta')\log\frac{1-\delta'}{1-\pp}\leq -\log(1-\delta')\leq 2(\log2) \delta'\,.$$
We have also $N_{k+2}\leq n ^{2/3}$, whence
$$\left\lfloor \frac{3^{d+1}n}{N_{k+2}}\right\rfloor\geq \frac{3^{d}n}{N_{k+2}}\,.$$
As $\delta'/\pp>8$, we have
 \begin{align*}
 \exp&\left(-(l_{k+2}l_{k+1})^d\left\lfloor \frac{3^{d+1}n}{N_{k+2}}\right\rfloor\left(\delta'\log\frac{\delta'}{\pp}+(1-\delta')\log\frac{1-\delta'}{1-\pp}\right)\right)\\
 &\hspace{1cm}\leq\exp\left(-(l_{k+2}l_{k+1})^d\left\lfloor \frac{3^{d+1}n}{N_{k+2}}\right\rfloor\left(3(\log 2)\delta'-2(\log 2)\delta'\right)\right)\\
 &\hspace{1cm}\leq \exp\left(-\frac{(3l_{k+2}l_{k+1})^d}{N_{k+2}}(\log2) \delta'n\right)\leq \exp\left(-(\log 2)\, \frac{\delta_k}{3R^d}n\right)\,.
 \end{align*}
Coming back to \eqref{eqcomp}, we have then
$$\Prb\big(n_k(\underline{\Gamma}_{(k)})\geq\delta_k n\big)\leq R^d\exp\left(-(\log 2)\, \frac{\delta_k}{3R^d}n\right)\,.$$
Finally, inequality \eqref{eq1} and the previous inequality yield
\begin{align}\label{eq14}
&\Prb\left(\begin{array}{c}\text{There exists a path $\gamma$ starting from $0$}\\\text{such that $|\gamma|\leq n$ and $n_k(\gamma)\geq \delta_k n$}\end{array}\right)\nonumber \\
&\hspace{2cm}\leq \sum_{\Gamma\in\animals\big(\big\lfloor \frac{3 ^{d+1}n}{N_{k+2}}\big\rfloor\big)} R^d\exp\left(-(\log 2)\,\frac{\delta_k}{3R^d}n\right)\nonumber\\
&\hspace{2cm}\leq R^d (7^d)^\frac{3^{d+1}n}{N_{k+2}}\exp\left(-(\log 2)\, \frac{\delta_k}{3 R^d}n\right)\,.
\end{align}
Since $d\geq 2$, we have
$$N_{k+2}=N_k^{4d^2}\geq N_{k}^{7d+2}=N_{k+1}^2N_k^{3d+1}N_k=N_k/\delta_k \geq l_0/\delta_k\,.$$
We can assume that $l_0$ satisfies furthermore
$$\frac{d(\log 7) \,3^{d+1}}{l_0}\leq \frac{\log 2}{6 R^d}\,.$$
If it is not the case we take a larger $l_0$. We have then
\begin{align}\label{eq14bis}
&\Prb\left(\begin{array}{c}\text{There exists a path $\gamma$ starting from $0$}\\\text{such that $|\gamma|\leq n$ and $n_k(\gamma)\geq \delta_k n$}\end{array}\right)\nonumber \\
&\hspace{4cm}\leq R^d\exp\left(- \frac{\log 2}{6R^dN_{k+1}^2N_k^{3d+1}}n\right)\,.
\end{align}
By construction, $N_{\overline{M}+1}=(N_{\overline{M}})^{2d}> n ^{1/3d}$ and so $N_{\overline{M}}>n^{1/6d^2}$. We are going to bound from above the number of bad $N_{\overline{M}}$-boxes that $\gamma$ crosses by the number of bad $N_{\overline{M}}$-boxes in the box $B_{4n}$.  Using theorem \ref{thm6}, we have
\begin{align}\label{eq13}
&\Prb\left(\begin{array}{c}\text{There exists a path $\gamma$ starting from $0$}\\\text{such that $|\gamma|\leq n$ and $n_{\overline{M}}(\gamma)\neq 0$}\end{array}\right)\nonumber \\
&\hspace{3cm}\leq \Prb\left(  \begin{array}{c}\text{There exists $\textbf{i}\in\sZ^d$ such that $B_{N_{\overline{M}}}(\textbf{i})\subset B_{4n}$}\\ \text{ and the box $B_{N_{\overline{M}}}(\textbf{i})$ is bad }\end{array}\right)\nonumber\\
&\hspace{3cm}\leq \left(\frac{4n+1}{2N_{\overline{M}}+1}\right)^d\exp\left(-C(p_0)\frac{N_{\overline{M}}}{(2R)^{d(\overline{M}-1)}}\right)\nonumber\\
&\hspace{3cm}\leq (4n+1)^d\exp\left(-C(p_0)\frac{n^{1/6d^2}}{(2R)^{d(\overline{M}-1)}}\right)\,.
\end{align}
Since $N_{\overline{M}}=l_0^{(2d)^{\overline{M}}}\leq n^{1/3d}$, there exist positive constants $C$ and $C'$ depending on $l_0$, $d$ and $p_0$ such that $$\overline{M}\leq C\log\log n\qquad \text{and}\qquad (2R)^{d(\overline{M}-1)}\leq (\log n) ^{C'}\,.$$ Thus, there exist positive constants $C_1$ and $C_2$ depending on $d$ and $p_0$ such that, for all $n\geq n_0$,
\begin{align}\label{eqM}
\Prb\left(\begin{array}{c}\text{There exists a path $\gamma$ starting from $0$ }\\\text{such that $|\gamma|\leq n$ and $n_{\overline{M}}(\gamma)\neq 0$}\end{array}\right)\leq C_1\exp\big(-C_2n^{1/(6d^2+1)}\big)\,.
\end{align}

\noindent We recall that $N_{k+1}\leq N_{\overline{M}}$ and $N_{k}=N_{k+1}^{1/2d}\leq N_{\overline{M}}^{1/2d} $. Since $N_{\overline{M}}\leq n ^{1/3d}$, combining inequalities \eqref{jsap}, \eqref{eq14bis} and \eqref{eqM}, we obtain that for $l_1\geq l_0$, for $n\geq n_0$,
\begin{align*}
&\Prb\left(\begin{array}{c}\text{There exists a path $\gamma$ starting from $0$ such that $|\gamma|\leq n$}\\\text{ and $\sum_{k=3}^{M(\gamma)-1} n_k(\gamma) N_{k+1}^2N_{k}^{3d}(3d)^{2k}\geq n$ or $N_{M(\gamma)}>n ^{1/3d}$}\end{array}\right)\nonumber\\
&\leq \sum_{k=3}^{\overline{M}-1} R^d \exp\left(- \frac{\log 2  }{6R^dN_{k+1}^2N_k^{3d+1}}n\right) + C_1\exp\left(-C_2n ^{1/(6d^2+1)}\right)\\
&\leq \sum_{k=3}^{\overline{M}-1} R^d \exp\left(- \frac{\log 2}{6R^d }\frac{n}{n^{2/3d+(3d+1)/6d^2 }}\right) + C_1\exp\left(-C_2n ^{1/(6d^2+1)}\right)\\
&\leq R^d C(\log\log n)\exp\left(- \frac{\log2}{6R^d}n^{1/3}\right)+C_1\exp\left(-C_2 n ^{1/(6d^2+1)}\right)\,.
\end{align*}
This yields the desired result.
\end{proof}

\begin{rk}\label{rkmultiscale}
 Let us briefly sketch our tentative to prove that the average bypass is at most a constant by using a simple renormalization process. Let $\gamma$ be the $q$-geodesic between $0$ and $nx$.
Following the proof of proposition \ref{propbypass}, the best control we can hope for is 
$$\sum_{e\in\gamma }|\shell(e)|\leq\sum_{\textbf{i}\in \sZ^d: C^{(1)}(\textbf{i})\cap\gamma \neq\emptyset}|C^{(1)}(\textbf{i})|^2\,.$$
 Since we have no control on the dependence between $\gamma$ and the states of the boxes, we must control the sum in the right hand side uniformly over all paths starting from $0$ of length at least $n$. Let $r$ be a fixed path starting from $0$ of length at least $n$.
 Since $|C^{(1)}(\textbf{0})|^2$ does not have an exponential moment, for every constant $c_0$, the following probability does not necessarily decay exponentially with $n$:
 $$\Prb\left (\sum_{\textbf{i}\in \sZ^d: C^{(1)}(\textbf{i})\cap r\neq\emptyset}|C^{(1)}(\textbf{i})|^2\geq c_0 n\right)\,.$$
 Hence, we do not have a sufficient control on this probability in order to counterbalance the combinatorial term coming from the possible choices for the path $r$.
\end{rk}

\section{Conclusion}\label{estimate}
In this section, we prove the main theorem \ref{thm1.2} and sketch the proof to extend this result in the case of general distributions.
\subsection{Bernoulli case: Proof of theorem \ref{thm1.2}}
Let $p_0>p_c(d)$.Let $\beta(p_0)$ given by theorem \ref{thm5}. Let $\delta_0(p_0)$ be given by theorem \ref{thm6}. Let $p,q$ be such that $p_0\leq p<q\leq p+\delta_0(p_0)$. Let $x\in\sZ^d\setminus\{0\}$. Let $(N_k)_{k\geq 1}$ be a sequence, $n_0$  and $M(\gamma)$ as given by proposition \ref{propcontmauvais}. Let us denote
$$\lambda=3^{2d}4d^2\beta\frac{(3d)^4N_{3}^2N_2^{2d}+ d}{p_0}\,.$$
Let $n\geq n_0$. Let us assume that $0$ and $nx$ belong to $\sC_{p_0}$. This implies that $0$ and $nx$ belong also to $\sC_q$. We denote by $\gamma$ a geodesic between $0$ and $nx$ in $\sC_q$, \textit{i.e.}, $\gamma$ is a $q$-open path such that $|\gamma|=D^{\sC_q }(0,nx)$. If there are several possible choices for $\gamma$, we choose one according to some deterministic rule. Using proposition  \ref{propcontmauvais}, we have
\begin{align}\label{eqcpr}
\Prb&\left(D^{\sC_p }\left(\widetilde{0}^{\sC_{p_0}},\widetilde{nx}^{\sC_{p_0}}\right)-D^{\sC_q }\left(\widetilde{0}^{\sC_{p_0}},\widetilde{nx}^{\sC_{p_0}}\right)\geq 25N_1\beta\lambda (q-p)\|x\|_1n\right)\nonumber\\
&\leq \Prb\left(\begin{array}{c}0\in\sC_{p_0},\,n x\in\sC_{p_0},|\gamma|\leq \beta n\|x\|_1,\\D^{\sC_p }\left(0,nx\right)-D^{\sC_q }\left(0,nx\right)\geq 25N_1\beta\lambda (q-p)\|x\|_1n\end{array}\right)\nonumber \\
&\hspace{0.5cm}+ (1-\Prb(0\in\sC_{p_0},\,n x\in\sC_{p_0}))+\Prb\left(\beta n\|x\|_1 <D^{\sC_q }(0,nx)<\infty \right)\nonumber\\
&\leq \Prb\left(\begin{array}{c}0\in\sC_{p_0},\,n x\in\sC_{p_0},|\gamma|\leq \beta n\|x\|_1,\\D^{\sC_p }\left(0,nx\right)-D^{\sC_q }\left(0,nx\right)\geq 25N_1\beta\lambda (q-p)\|x\|_1n,\\\sum_{k=3}^{M(\gamma)-1} n_k(\gamma) N_{k+1}^2N_{k}^{3d}(3d)^{2k}< \beta n \|x\|_1, \,N_{M(\gamma)}\leq n^{1/3d}\end{array}\right) \nonumber\\
&\hspace{0.5cm}+ (1-\Prb(0\in\sC_{p_0},\,n x\in\sC_{p_0}))+\Prb\left(\beta n\|x\|_1 <D^{\sC_q }(0,nx)<\infty \right)\nonumber\\
&\hspace{0.5cm}+\exp(-A_1n^{1/(6d^2+1)})\,.
\end{align}
Using the FKG inequality, we have
\begin{align}\label{eqcpr4}
\Prb(0\in\sC_{p_0},\, nx\in\sC_{p_0})\geq \Prb(0\in\sC_{p_0})^2>0\,.
\end{align}
Using lemma \ref{AP}, we have
\begin{align}\label{eq:eq20}
\Prb\left(\beta n\|x\|_1 <D^{\sC_q }(0,nx)<\infty \right)\leq\widehat{A}\exp(-\widehat{B}n)\,.
\end{align}
We set $$\widetilde{\gamma}=\gamma \setminus  (B_{4n ^{1/3d}}\cup( B_{4n ^{1/3d}}+nx))\,.$$ Note that on the event $\{N_{M(\gamma)}\leq n ^{1/3d}\}$, we have $\widetilde{\gamma}\subset\overline{\gamma}$.
On the event \[\left\{\sum_{k=3}^{M(\gamma)-1} n_k(\gamma) N_{k+1}^2N_{k}^{3d}(3d)^{2k}\leq \beta n \|x\|_1,\,|\gamma|\leq\beta n\|x\|_1\right\}\,,\] using proposition \ref{propconstr}, we obtain
\begin{align*}
\sum_{e\in\widetilde{\gamma}}|\shell(e)|^2&\leq 3^{2d}4d^2\left((3d)^4|\gamma|N_{3}^2N_2^{2d}+ \sum_{k=3}^{M(\gamma)-1} n_k(\gamma)N_{k+1}^2N_k^{3d}(3d)^{2k}d\right)\\
&\leq 3^{2d}4d^2\beta\|x\|_1\left((3d)^4N_{3}^2N_2^{2d}+ d\right)n\leq \lambda \|x\|_1 p_0 n\,.
\end{align*}
Hence
\begin{align}
\Prb&\left(\begin{array}{c}0\in\sC_{p_0},\,n x\in\sC_{p_0},|\gamma|\leq \beta n\|x\|_1,\\D^{\sC_p }\left(0,nx\right)-D^{\sC_q }\left(0,nx\right)\geq 25N_1\beta\lambda (q-p)\|x\|_1n,\\\sum_{k=3}^{M(\gamma)-1} n_k(\gamma) N_{k+1}^2N_{k}^{3d}(3d)^{2k}< \beta n \|x\|_1, \,N_{M(\gamma)}\leq n ^{1/3d}\end{array}\right) \nonumber\\
&\qquad\leq \,\Prb\left(\begin{array}{c}0\in\sC_{p_0},\,n x\in\sC_{p_0},|\gamma|\leq \beta n\|x\|_1,\\D^{\sC_p }\left(0,nx\right)-D^{\sC_q }\left(0,nx\right)\geq 25N_1\beta\lambda (q-p)\|x\|_1n,\\\sum_{e\in\widetilde{\gamma}}|\shell(e)|^2\leq \lambda p_0 \|x\|_1 n , \,N_{M(\gamma)}\leq n ^{1/3d}\end{array}\right) \,.
\end{align}
We would like to introduce a coupling of the percolation processes with parameters $q$ and $p$ such that, if an edge is $p$-open, then it is $q$-open, and we would like that the random path $\gamma$ and the shells associated to the edges in $\gamma$ are in some sense independent from the $p$-state of the edges in $\widetilde{\gamma}$. This is not the case when we use the classical coupling with a unique uniform random variable associated to each edge. For a fixed path $r$, given a family of shells associated with the edges of $r$ and a subset $E$ of $\widetilde{r}$, we do not use the $p$-state of the edges in $r$ to build the bypasses of the edges in $E$, because the existence of $p$-open bypasses for the edges in $E$ does not depend on the $p$-states of the edges in $r$. Indeed, a bypass is a $p$-open path of edges which connects two vertices of $r$ but which does not go through an edge of $r$. To clarify the computations, we introduce three sources of randomness in order to ensure that the choice of $\gamma$ and its shells are independent from the $p$-states of the edges in $\widetilde{\gamma}$. To each edge we associate three independent Bernoulli random variables $V$, $W$ and $Z$ of respective parameters $q$, $p/q$ and $p/q$. The random variables $Z V$ and $ZW$ are Bernoulli random variables of parameter $p$ and we have
\begin{align*}
\Prb(ZV=0\,|\,V=1)=\Prb(Z=0\,|\,V=1)=\Prb(Z=0)=1-\frac{p}{q}=\frac{q-p}{q}\,.
\end{align*}
The $q$-states of the edges is given by the family $(V_e)_{e\in\E^d}$. Once we know the $q$-states of the edges, we find a geodesic $\gamma$. The $p$-states of the edges outside $\gamma$ is given by the family $(V_eZ_e)_{e\in\E^d}$, while the $p$-states of the edges belonging  to $\gamma$ is given by the family $(V_eW_e)_{e\in\E^d}$. More precisely, on the event $\{\gamma=r\}$, an edge $e\in\widetilde{r}$ is $p$-open if $W_e=1$, whereas an edge $e$ in $\E^d\setminus \widetilde{r}$ is $p$-open if $V_e Z_e=1$.
We denote by $\cE$ the set of the $p$-closed edges in $\widetilde{\gamma}$. The event $\{\gamma=r\}$ depends only on the family $(V_e)_{e\in\E^d}$, the event $\{\shell(e)=S_e\}$ depends on the families of random variables $(V_e)_{e\in\E^d}$ and $(Z_e)_{e\in \E^d}$. Finally the event $\{\cE=E\}$ depends on the random variables  $(W_e)_{e\in\widetilde{r}}$. 
We sum over all possible realizations of $\gamma$, $(\shell(e))_{e\in\widetilde{\gamma}}$ and $\cE$:
\begin{align}\label{eqcpr2}
\Prb&\left(\begin{array}{c}0\in\sC_{p_0},\,n x\in\sC_{p_0},|\gamma|\leq \beta n\|x\|_1,\\D^{\sC_p }\left(0,nx\right)-D^{\sC_q }\left(0,nx\right)\geq 25N_1\beta\lambda (q-p)\|x\|_1n,\\\sum_{e\in\widetilde{\gamma}}|\shell(e)|^2\leq \lambda p_0 \|x\|_1n,\,N_{M(\gamma)}\leq n ^{1/3d}\end{array}\right)\nonumber\\
&=\sum_{\substack{r \text{ path}\\|r|\leq \beta n\|x\|_1}}\sum_{\substack{S_e,e\in\widetilde{r}\\\sum_{e\in\widetilde{r}}|S_e|^2\\\leq \lambda p_0\|x\|_1 n}}\sum_{E\subset \widetilde{r}}\Prb\left(\begin{array}{c}\gamma=r,\,\cE=E,\,\forall e\in\widetilde{r}\quad \shell(e)=S_e\\D^{\sC_p }\left(0,nx\right)-|r|\geq 25N_1\beta\lambda (q-p)\|x\|_1n,\\N_{M(\gamma)}\leq n ^{1/3d}\end{array}\right)\,.
\end{align}
By proposition \ref{propbypass}, we can build a path $r'$ such that $r'$ does not contain any edge of $E$, the edges in $r'\setminus r$ are $p$-open and $$|r'\setminus r|\leq 12\beta N_1\sum_{e\in E }|S_e|\,.$$ We recall that we work on the event $\{N_{M(\gamma)}\leq n ^{1/3d}\}$. At this stage, the edges in $$r'\cap (B_{4n ^{1/3d}}\cup( B_{4n ^{1/3d}}+nx))$$ are not necessarily $p$-open. We denote by $y$ (respectively $z$) the first intersection of $r'$ with $\partial B_{4n^{1/3d}}$ (respectively the last intersection of $r'$ with $\partial B_{4n^{1/3d}}+nx$). Let $C_p(w)$ denotes the $p$-open cluster of $w\in\sZ^d$. From the previous construction, we see that $C_p(y)$ and $C_p(z)$ have cardinality at least $n-8n ^{1/3d}$. Using the result of Kesten and Zhang in \cite{Kesten1990}, we get
\begin{align}\label{eqcpr6}
\Prb(y\notin\sC_p)&\leq \Prb(n-4n^{1/3d}<|C_p(y)|<\infty)\nonumber\\
&\leq \sum_{y\in \partial B_{4n^ {1/3d}}}\Prb(n-8n^{1/3d}<|C_p(y)|<\infty)\nonumber\\
&\leq  |\partial B_{4n^ {1/3d}}|C_1\exp(-C_2n^{(d-1)/d})\,.
\end{align}
Therefore, with probability at least $1-2 |\partial B_{4n^ {1/3d}}|C_1\exp(-C_2n^{(d-1)/d})$, the vertices $y$ and $z$ belong to $\sC_p$.
Applying lemma \ref{AP}, we have
\begin{align}\label{eqcpr5}
\Prb&(\exists y\in \partial B_{4n^ {1/3d}},\, \beta \|y\|_1 \leq D^{\sC_p}(0,y)<\infty)\nonumber\\
&\leq \sum_{y\in \partial B_{4n^ {1/3d}}}\Prb( \beta \|y\|_1 \leq D^{\sC_p}(0,y)<\infty)\leq |\partial B_{4n^ {1/3d}}|\widehat{A}\exp(-2\widehat{B}n^{1/3d})\,.
\end{align}
Thus with probability at least $1-2|\partial B_{4n^ {1/3d}}|\widehat{A}\exp(-2\widehat{B}n^{1/3d})$, we can join $0$ and $y$ (respectively $z$ and $nx$) by a $p$-open path $r^{first}$ (respectively $r^{last}$) of length at most $4d\beta n ^{1/3d}$. By concatenating $r^{first}$, the portion of $r'$ between $y$ and $z$, and $r^{last}$ in this order, we obtain a $p$-open path $r''$ that joins $0$ and $nx$ such that
\begin{align}\label{eqcpr3}
|r''\setminus r|\leq |r^{first}|+|r^{last}|+|r'\setminus r|\leq 8d\beta n^{1/3d}+ 12\beta N_1\sum_{e\in E }|S_e|\,.
\end{align}
Therefore, combining inequalities \eqref{eqcpr2}, \eqref{eqcpr6}, \eqref{eqcpr5} and \eqref{eqcpr3}, we get for $n$ large enough
\begin{align}\label{eqsum}
\Prb&\left(\begin{array}{c}0\in\sC_{p_0},\,n x\in\sC_{p_0},|\gamma|\leq \beta n\|x\|_1,\\D^{\sC_p }\left(0,nx\right)-D^{\sC_q }\left(0,nx\right)\geq 25N_1\beta\lambda (q-p)\|x\|_1n,\\\sum_{e\in\widetilde{\gamma}}|\shell(e)|^2\leq \lambda p_0 \|x\|_1n,\,N_{M(\gamma)}\leq n ^{1/3d}\end{array}\right)\nonumber\\
&\leq\sum_{\substack{r \text{ path}\\|r|\leq \beta n\|x\|_1}}\sum_{\substack{S_e,e\in\widetilde{r}\\\sum_{e\in\widetilde{r}}|S_e|^2\\\leq \lambda p_0\|x\|_1 n}}\sum_{E\subset \widetilde{r}}\Prb\left(\begin{array}{c}\gamma=r,\,\cE=E,\,\forall e\in\widetilde{r}\quad \shell(e)=S_e\\12 N_1\sum_{e\in E }|S_e|\geq 24 N_1(q-p)\lambda \|x\|_1 n\end{array}\right)\nonumber\\
&\hspace{0.5cm}+ 2 |\partial B_{4n^ {1/3d}}|\big(C_1\exp(-C_2n^{(d-1)/d})+ \widehat{A}\exp(-2\widehat{B}n^{1/3d})\big)\,.
\end{align}
Using the definition of our coupling, we get
\begin{align}\label{eqcouplage}
&\sum_{\substack{r \text{ path}\\|r|\leq \beta n\|x\|_1}}\sum_{\substack{S_e,e\in\widetilde{r}\\\sum_{e\in\widetilde{r}}|S_e|^2\leq \lambda p_0\|x\|_1 n}}\sum_{E\subset \widetilde{r}}\Prb\left(\begin{array}{c}\gamma=r,\,\cE=E,\,\forall e\in\widetilde{r}\quad \shell(e)=S_e\\\sum_{e\in \cE}|S_e|\geq 2\lambda(q-p)\|x\|_1 n\end{array}\right)\nonumber\\
&\hspace{1cm}\leq\sum_{\substack{r \text{ path}\\|r|\leq \beta n\|x\|_1}}\sum_{\substack{S_e,e\in\widetilde{r}\\\sum_{e\in\widetilde{r}}|S_e|^2\\\leq \lambda p_0\|x\|_1 n}}\Prb\left(\begin{array}{c}\gamma=r,\,\forall e\in\widetilde{r}\quad \shell(e)=S_e\\\sum_{e\in \widetilde{r} }(1-W_e)|S_e|\geq 2\lambda(q-p)\|x\|_1 n\end{array}\right)\nonumber\\
&\hspace{1cm}=\sum_{\substack{r \text{ path}\\|r|\leq \beta n\|x\|_1}}\sum_{\substack{S_e,e\in\widetilde{r}\\\sum_{e\in\widetilde{r}}|S_e|^2\\\leq \lambda p_0\|x\|_1 n}}\Prb\Big(\gamma=r,\,\forall e\in\widetilde{r}\quad \shell(e)=S_e\Big)\nonumber\\
&\hspace{4cm}\times\Prb\left(\sum_{e\in \widetilde{r} }(1-W_e)|S_e|\geq 2\lambda(q-p)\|x\|_1 n\right)\,.
\end{align}
In the last step, we used the fact that the random variables $(W_e,e\in\widetilde{r})$ are independent from the event $\{\gamma=r\}$ and the shells $(\shell(e),e\in\widetilde{r})$.
Let us set $$U=\sum_{e\in \widetilde{r} }(1-W_e)|S_e|\,.$$ We have
\begin{align}\label{eqesp}
\E\left(U\right)=\frac{q-p}{q}\sum_{e\in \widetilde{r} }|S_e|\leq \lambda(q-p)\|x\|_1 n
\end{align}
and
\begin{align}\label{eqvar}
\Var\left(U\right)\leq \Var(1-W)\sum_{e\in\widetilde{r}}|S_e|^2\leq \lambda\frac{p_0}{4}\|x\|_1n\,.
\end{align}
Using the Markov inequality and inequalities \eqref{eqesp} and \eqref{eqvar}, we get
\begin{align}\label{eqMarkov}
\Prb\left(\sum_{e\in \widetilde{r} }(1-W_e)|S_e|\geq 2\lambda(q-p)\|x\|_1 n\right)&\leq \Prb\big(|U-\E(U)|\geq \lambda(q-p)\|x\|_1 n\big)\nonumber\\
&\leq \frac{\Var (U)}{\left(\lambda(q-p)\|x\|_1 n\right)^2}\leq \frac{p_0}{4\lambda(q-p)^2\|x\|_1n}\,.
\end{align}
Finally, combining inequalities \eqref{eqcpr}, \eqref{eqcpr4}, \eqref{eq:eq20}, \eqref{eqsum},  \eqref{eqcouplage} and \eqref{eqMarkov}, we deduce the existence of a real number $\pp(p_0,p,q)>0$ such that, for $n$ large enough,
$$\Prb\left(D^{\sC_p }\left(\widetilde{0}^{\sC_{p_0}},\widetilde{nx}^{\sC_{p_0}}\right)-D^{\sC_q }\left(\widetilde{0}^{\sC_{p_0}},\widetilde{x}^{\sC_{p_0}}\right)\leq 25N_1\beta\lambda (q-p)\|x\|_1n\right)\geq\pp(p_0,p,q)\,.$$
Let $\delta>0$. Thanks to the convergence of the regularized times given by proposition \ref{convergence}, we can also choose $n$ large enough such that
$$\Prb\left(\mu_p(x)-\delta \leq \frac{D^{\sC_p}(\widetilde{0}^{\sC_{p_0}},\widetilde{nx}^{\sC_{p_0}})}{n}\right)\geq 1-\frac{\pp(p_0,p,q)}{3}\,,$$
$$\Prb\left(  \frac{D^{\sC_q}(\widetilde{0}^{\sC_{p_0}},\widetilde{nx}^{\sC_{p_0}})}{n}\leq\mu_q(x)+\delta\right)\geq 1-\frac{\pp(p_0,p,q)}{3}\,.$$
The intersection of the three previous events has positive probability. On this intersection, we have
\begin{align*}
\mu_p(x)-\delta\leq \mu_q(x)+\delta+25N_1\beta\lambda (q-p)\|x\|_1\,.
\end{align*}
This inequality occurs with positive probability, yet all the quantities in it are deterministic. By taking the limit when $\delta$ goes to $0$, we get
\begin{align*}
\mu_p(x)\leq \mu_q(x)+25N_1\beta\lambda (q-p)\|x\|_1\,. 
\end{align*}
Thus for all $p\geq p_0$ and $p<q\leq p+\delta_0$, there exists a positive constant $C'(p_0)$ such that $$\forall x\in\sZ^d\qquad\mu_p(x)-\mu_q(x)\leq C'(p_0)(q-p)\|x\|_1\,.$$
We recall that the map $p\mapsto \mu_p$ is non-increasing. We consider now the case $q>p+\delta_0$.
We write $q-p=k\delta_0 + r$ with $k\in\sN$ and $0\leq r<\delta_0$. 
We obtain
\begin{align}\label{eq6.1}
\mu_p(x)-\mu_q(x)&=\sum_{i=0}^{k-1}\mu_{p+i\delta_0}(x)-\mu_{p+(i+1)\delta_0}(x)+\mu_{q-r}(x)-\mu_q(x)\nonumber\\
&\leq\, \sum_{i=0}^{k-1}C'(p_0)\delta_ 0\|x\|_1+ C'(p_0)r\|x\|_1=C'(p_0)(k\delta_0+r)\|x\|_1\nonumber\\
&=C'(p_0)(q-p)\|x\|_1\,.
\end{align}
By homogeneity, \eqref{eq6.1} also holds for all $x\in\mathbb{Q}^d$. Let us recall that for all $x,y\in\mathbb{R}^d$ and $p\geq p_c(d)$, we have (see for instance theorem 1 in \cite{cerf2016})
\begin{align}\label{eqcerf}
|\mu_p(x)- \mu_p(y)|\leq \mu_p(e_1)\|x-y\|_1\,.
\end{align}
Moreover, by compactness of $\sS^{d-1}$, there exists a finite set $(y_1,\dots,y_m)$ of rational points of $\mathbb{S}^{d-1}$ such that 
\begin{align*}
\mathbb{S}^{d-1}\subset \bigcup_{i=1}^m\Big\{\,x\in\mathbb{S}^{d-1}:\|y_i-x\|_1\leq (q-p)\,\Big\}\, .
\end{align*}
Let $x\in\mathbb{S}^{d-1}$ and $y_i$ such that $\|y_i-x\|_1\leq (q-p)$, using inequality \eqref{eqcerf}, we get
\begin{align*}
|\mu_p(x)- \mu_q(x)|&\leq |\mu_p(x)- \mu_p(y_i)|+|\mu_p(y_i)- \mu_q(y_i)|+|\mu_q(y_i)- \mu_q(x)|\\
&\leq \mu_p(e_1)\|y_i-x\|_1+   C'(p_0)(q-p)  + \mu_q(e_1)\|y_i-x\|_1\\
&\leq \left(2\mu_{p_0}(e_1)+C'(p_0) \right)(q-p),
\end{align*}
where we use the monotonicity of the map $p\rightarrow\mu_p$ in the last inequality.
Finally, for any $p,q\in[p_0,1]$,
$$\sup_{x\in\sS^{d-1}}|\mu_p(x)-\mu_q(x)|\leq \left(2\mu_{p_0}(e_1)+C'(p_0) \right)|q-p|\,.$$
This yields the result.
\subsection{General distributions case}\label{sec:gendib}
In this section, we sketch the proof for the general distributions case.
The following proposition relies on a result of Kesten (proposition (5.8) in \cite{Kesten:StFlour}).
\begin{prop}\label{propconstante} Let $p_1<p_c(d)$, $p_0>p_c(d)$, $M>0$, $\ep_0>0$ and $\ep\mapsto\delta(\ep)$ be a non-decreasing function such that $\delta(\ep_0)\leq 1-p_1$. Let $\fC_{p_0,p_1,M,\ep_0,\delta}$ be the class of functions defined in the statement of the result.
There exist positive constants $A$, $B$ and $C$ depending on these parameters such that, for any distribution $G$ in $\fC_{p_0,p_1,M,\ep_0,\delta}$,
$$\forall n\geq 1\qquad\Prb\left(\begin{array}{c}\text{There exists a path $r$ starting from $0$}\\\text{ such that $|r|\geq n$ and $T_G(r)<Cn$}\end{array}\right)\leq A\exp(-Bn)\,.$$
\end{prop}
\begin{proof}[Proof of proposition \ref{propconstante}]
Let $H$ be the distribution such that $H(\{0\})=p_1$, for every $\ep\leq\ep_0$, we have $H(]0,\ep])=\delta(\ep)$ and $H(\{\ep_0\})=1-H([0,\ep_0[)$. Using proposition 5.8 in \cite{Kesten:StFlour}, there exist positive constants $A$, $B$ and $C$ depending on the distribution $H$ such that
$$\forall n\geq 1\qquad\Prb\left(\begin{array}{c}\text{There exists a path $r$ starting from $0$}\\\text{ such that $|r|\geq n$ and $T_H(r)<Cn$}\end{array}\right)\leq A\exp(-Bn)\,.$$ 
By construction, any distribution $G$ in $\fC_{p_0,p_1,M,\ep_0,\delta}$ stochastically dominates $H$. Let $n\geq 1$. Using the stochastic domination, we have
\begin{align*}
\Prb&\left(\begin{array}{c}\text{There exists a path $r$ starting from $0$}\\\text{ such that $|r|\geq n$ and $T_G(r)<Cn$}\end{array}\right)\\
&\hspace{1cm}\leq \Prb\left(\begin{array}{c}\text{There exists a path $r$ starting from $0$}\\\text{ such that $|r|\geq n$ and $T_H(r)<Cn$}\end{array}\right)\leq A\exp(-Bn)\,.
\end{align*}
This yields the result.
\end{proof}
The proof of the result for general distributions uses the same strategy as the proof of theorem \ref{thm1.2}, so we will only sketch the arguments.
\begin{proof}
Let $p_1<p_c(d)$, $p_0>p_c(d)$ $M>0$, $\ep_0>0$ and $\ep\mapsto\delta(\ep)$ be a function.
Let $G$ and $F$ be two distributions on $[0,+\infty[$ in $\fC_{p_0,p_1,M,\ep_0,\delta}$. Let $p<q\in[p_0,1]$. We define $G_p$ and $F_q$ as $$G_p=p\, G+(1-p)\delta_\infty\quad\text{and}\quad F_q=q\,F+(1-q)\delta_\infty\,.$$
We write $q-p=k\delta_0/2 + r$ with $k\in\sN$ and $0\leq r<\delta_0/2$. 
We have
\begin{align}\label{eqdecomp}
|\mu_{F_q}(x)&-\mu_{G_p}(x)|\nonumber\\
&\leq |\mu_{F_q}(x)-\mu_{G_{q-r}}(x)|+\sum_{i=0}^{k-1}|\mu_{G_{p+(i+1)\delta_0/2}}(x)-\mu_{G_{p+i\delta_0/2}}(x)|\,.
\end{align}
 Let \smash{$i\in\{0,\dots,k-1\}$}. We introduce next a coupling between \smash{$G_{p+i\delta_0/2}$} and \smash{$G_{p+(i+1)\delta_0/2}$}. To each edge $e\in\E^d$, we associate three random variables: one uniform random variable $U(e)$ on $[0,1]$, two Bernoulli random variables  $V(e)$ and $W(e)$ of parameters $p+(i+1)\delta_0/2$ and $\delta_0/2$. We set $$t_{{G_{p+(i+1)\delta_0/2}}}(e)=\left\{
    \begin{array}{ll}
        G^{-1}(U(e))& \mbox{if} \quad V(e)=1\\
        +\infty & \mbox{otherwise}\,\, 
    \end{array}
\right.$$
and 
$$t_{{G_{p+i\delta_0/2}}}(e)=\left\{
    \begin{array}{ll}
        G^{-1}(U(e))& \mbox{if} \quad V(e)W(e)=1\\
        +\infty & \mbox{otherwise},\, 
    \end{array}\,.
\right.$$
With this coupling, we have 
$$\forall e\in\E^d\quad t_{{G_{p+i\delta_0/2}}}(e)\geq t_{{G_{p+(i+1)\delta_0/2}}}(e),$$
thus $\mu_{{G_{p+i\delta_0/2}}}$ stochastically dominates $\mu_{{G_{p+(i+1)\delta_0/2}}}$.
Since $G\in \fC_{p_0,p_1,M,\ep_0,\delta}$, we have $G([0,M])\geq 1-\delta_0/2$ and therefore
\begin{align*}
G_{p+(i+1)\delta_0/2}([0,\infty[)&- G_{p+i\delta_0/2}([0,M])\nonumber\\
&\leq p+(i+1)\frac{\delta_0}{2}- \left(p+i\frac{\delta_0}{2}\right)\left(1-\frac{\delta_0}{2}\right)\leq \delta_0\,.
\end{align*}
Let $n\geq 1$. Let us assume that $0$ and $nx$ belong to $\sC_{p_0,M}$, the infinite cluster made of edges $e$ such that $t_{G_{p_0}}(e)\leq M$. Let $\gamma$ be a geodesic between $0$ and $nx$ for the passage times $(t_{G_{p+(i+1)\delta_0/2}}(e))_{e\in\E^d}$. For some edges $e\in\gamma$, we have $t_{G_{p+i\delta_0/2}}(e)=\infty$. We would like to bypass these edges using only edges \smash{$e\in\E^d$ }such that $t_{G_{p+i\delta_0/2}}(e)\leq M$. We control the length of the bypasses using the same strategy as in the Bernoulli case.  Up to choosing a smaller $\delta_0$, we can assume that
$$ G_{p+i\delta_0/2}([0,M])=\left(p+i\frac{\delta_0}{2}\right)\left(1-\frac{\delta_0}{2}\right)\geq p_0 \left(1-\frac{\delta_0}{2}\right)>p_c(d)\,.$$ We say that the edge $e$ is $G_{p+i\delta_0/2}([0,M])$-open (respectively $G_{p+(i+1)\delta_0/2}([0,\infty[)$-open) if $t_{G_{p+i\delta_0/2}}(e)\leq M$ (respectively $t_{G_{p+(i+1)\delta_0/2}}(e)<\infty$). The shells are now made of $(G_{p+i\delta_0/2}([0,M]),G_{p+(i+1)\delta_0/2}([0,\infty[))$-good boxes. Let $(\shell(e))_{e\in\gamma}$ be a family of shells as in proposition \ref{propconstr}. We can build $\gamma'$ and $\gamma''$ as in the Bernoulli case, where $\gamma''$ is a path between $0$ and $nx$ whose edges have finite passage times for the distribution $G_{p+i\delta_0/2}$. With this coupling, the passage times coincide for the two distributions on $\gamma''\cap \gamma$. Thus, we have
\begin{align*}
T&_{G_{p+i\delta_0/2}}(\gamma'')\\
&\leq T_{G_{p+i\delta_0/2}}(\gamma ^{(first)})+ T_{G_{p+i\delta_0/2}}(\gamma ^{(last)})+T_{G_{p+i\delta_0/2}}(\gamma\cap \gamma'')+T_{G_{p+i\delta_0/2}}(\gamma'\setminus \gamma)\\
&\leq 4dM\beta n^{1/3d}+T_{G_{p+(i+1)\delta_0/2}}(\gamma)+12\beta N_1M\sum_{e\in\overline{\gamma}}|\shell(e)|\mathds{1}_{W(e)=0}\,.
\end{align*}
We need then to control the length of $\gamma$.
We have
$$G_{p+(i+1)\delta_0/2}([0,M])\geq p_0 \,.$$
Since $0$ and $nx$ belong to $\sC_{p_0,M}$, then $0$ and $nx$ also belong to $\sC_{p+(i+1)\delta_0/2,M}$, the infinite cluster made of edges $e$ such that $t_{G_{p+(i+1)\delta_0/2}}(e)\leq M$. Using lemma~\ref{AP}, we obtain that, with high probability,
$$D^{\sC_{p+(i+1)\delta_0/2,M}}(0,nx)\leq\beta n \|x\|_1,$$ which implies further that $$ T_{G_{p+(i+1)\delta_0/2}}(\gamma)\leq M \beta n \|x\|_1\,.$$
Since $G_{p+(i+1)\delta_0/2}$ belongs to $\fC_{p_0,p_1,M,\ep_0,\delta}$, we have, with high probability, 
$$|\gamma|\leq  \frac{M}{C} \beta n \|x\|_1\,,$$
where $C$ is the constant defined in proposition \ref{propconstante} corresponding to the class $\fC_{p_0,p_1,M,\ep_0,\delta}$.
As in the Bernoulli case, we have, with high probability,
$$\sum_{e\in\overline{\gamma}}|\shell(e)|\mathds{1}_{W(e)=0}\leq \frac{\delta_0}{2}\frac{M}{C} A'(p_0)\beta n \|x\|_1\,.$$
In the same way than in the Bernoulli case, we conclude that there exists a positive constant $\kappa$ that depends on $p_1$, $p_0$, $\ep_0$, $M$, $\delta$ and $d$ such that 
$$\sup_{x\in\sS^{d-1}}|\mu_{G_{p+i\delta_0/2}}(x)-\mu_{G_{p+(i+1)\delta_0/2}}(x)|\leq \kappa \frac{\delta_0}{2}\,.$$
Therefore 
$$\sum_{i=0}^{k-1}|\mu_{G_{p+i\delta_0/2}}(x)-\mu_{G_{p+(i+1)\delta_0}/2}(x)|\leq \kappa\,k\frac{\delta_0}{2}\,.$$
The last step of the proof consists in controlling the quantity $|\mu_{F_q}(x)-\mu_{G_{q-r}}(x)|$. We use again the same strategy. We introduce a specific coupling. To each edge $e\in\E^d$, we associate three random variables: a uniform random variable $U(e)$ on $[0,1]$, two Bernoulli random variables  $V(e)$ and $W(e)$ of parameters $q$ and $r$. We set $$t_{F_q}(e)=\left\{
    \begin{array}{ll}
        F^{-1}(U(e))& \mbox{if} \quad V(e)=1\\
        +\infty & \mbox{otherwise}\,\, 
    \end{array}
\right.$$
and 
$$t_{{G_{q-r}}}(e)=\left\{
    \begin{array}{ll}
        G^{-1}(U(e))& \mbox{if} \quad V(e)W(e)=1\\
        +\infty & \mbox{otherwise}.\,\, 
    \end{array}
\right.$$
We consider the geodesic $\gamma$ between $0$ and $nx$ for the passage times $(t_{F_q}(e))_{e\in\E^d}$. For some edges $e\in\gamma$, we have $t_{G_{q-r}}(e)=\infty$. We shall bypass these edges using only edges $e\in\E^d$ such that $t_{G_{q-r}}(e)\leq M$. We build $\gamma''$ as before. The main difference is that, for edges in $\gamma''\cap\gamma$, the passage times for the distributions $F_q$ and $G_{q-r}$ do not coincide any more. For $e\in\gamma''\cap\gamma$, we have
$$|t_{F_q}(e)-t_{G_{q-r}}(e)|=\big|F^{-1}(U(e))-G^{-1}(U(e))\big|\leq\sup_{t\in[0,1]}\big|F^{-1}(t)-G^{-1}(t)\big|\,.$$ 
Thus, we obtain
\begin{align*}
&T_{G_{q-r}}(0,nx)\leq T_{G_{q-r}}(\gamma'')\leq T_{G_{q-r}}(\gamma''\cap \gamma)+T_{G_{q-r}}(\gamma''\setminus \gamma)\\
&\quad\leq T_{F_q}(\gamma)+ |\gamma|\sup_{t\in[0,1]}|F^{-1}(t)-G^{-1}(t)|+12\beta N_1 M\sum_{e\in \overline{\gamma}}|\shell(e)|\ind_{W(e)=0}\,,
\end{align*}
where the shells are made of $(G_{q-r}([0,M]),F_{q}([0,\infty[))$-good boxes.
We can show as above that there exists a positive constant $\kappa'$ depending on the parameters of the class $\fC$ such that
\begin{align*}
\sup_{x\in\sS^{d-1}}\left(\mu_{G_{q-r}}(x)-\mu_{F_q}(x)\right)\leq \kappa'\left(\sup_{t\in[0,1]}|F^{-1}(t)-G^{-1}(t)|+r\right)\,.
\end{align*}
To prove the converse inequality, we consider the geodesic $\pi$ between $0$ and $nx$ for the law $G_{q-r}$. Given the coupling, any edge in $\pi$ has finite passage time for the law $F_q$. Therefore, we have
\begin{align}
T_{F_q}(0,nx)\leq T_{F_q}(\pi)\leq T_{G_{q-r}}(\pi)+|\pi|\sup_{t\in[0,1]}|F^{-1}(t)-G^{-1}(t)|\,.
\end{align}
We obtain the converse inequality and therefore
\begin{align}\label{eqder}
\sup_{x\in\sS^{d-1}}|\mu_{G_{q-r}}(x)-\mu_{F_q}(x)|\leq \kappa'\left(\sup_{t\in[0,1]}|F^{-1}(t)-G^{-1}(t)|+r\right)\,.
\end{align}
Finally, combining inequalities \eqref{eqdecomp} and \eqref{eqder}, we get
$$\sup_{x\in\sS^{d-1}}|\mu_{F_q}(x)-\mu_{G_p}(x)|\leq \max(\kappa,\kappa')\left( |q-p|+\sup_{t\in[0,1]}|F^{-1}(t)-G^{-1}(t)|\right)\,.$$
This yields the result.
\end{proof}
\bibliographystyle{plain}
\bibliography{biblio}

\end{document}